\documentclass[12pt]{article}
\usepackage{verbatim}
\usepackage{fancyhdr}
\usepackage{graphicx}
\usepackage{epstopdf}
\usepackage{amsmath}
\usepackage{float}
\usepackage{amsmath}
\allowdisplaybreaks[4]
\usepackage{amsthm}
\usepackage{amssymb}
\usepackage{bm}
\usepackage{eucal}
\usepackage{color}
\usepackage[table]{xcolor}
\usepackage[figureright]{rotating}
\usepackage[titletoc]{appendix}
\usepackage[figureright]{rotating}
\usepackage{multirow}
\usepackage{subfigure}
\usepackage{tikz}
\usepackage{pgfplots}
\usepgfplotslibrary{groupplots}

\graphicspath{ {Figures/} }
\newtheoremstyle{plainNoItalics}{}{}{\normalfont}{}{\bfseries}{.}{ }{}

\theoremstyle{plain}
\newtheorem{thm}{Theorem}[section]

\theoremstyle{plainNoItalics}

\newtheorem{lem}[thm]{Lemma}

\newtheorem{rem}[thm]{Remark}

\numberwithin{equation}{section}

\newcommand{\inj}{\int_{I_j\cap\Omega}}

\newcommand{\ino}{\int_{\Omega}}
\newcommand{\V}{\mathcal{V}_h^r}

\newcommand{\jl}{{j-\frac{1}{2}}}

\newcommand{\p}{\pi_h^e}

\newcommand{\Ij}{{_{I_j\cap\Omega}}}
\newcommand{\rM}{\gamma_M}
\newcommand{\rA}{\gamma_A}

\newcommand{\Th}{\mathcal{T}_h}
\newcommand{\Fh}{\mathcal{F}_h}
\newcommand{\Fr}{\mathcal{F}_\Gamma}
\newcommand{\Tr}{\mathcal{T}_{\Gamma}}

\allowdisplaybreaks[4]

\title{High order cut discontinuous Galerkin methods for hyperbolic conservation laws in one space dimension 
}
\author{Pei Fu\footnote{ Division of Scientific Computing, Department of Information Technology, Uppsala University, 75237 Uppsala, Sweden. Email: pei.fu@it.uu.se, gunilla.kreiss@it.uu.se.},
Gunilla Kreiss \footnotemark[1] 
}

\date{}
\begin{document}
\maketitle
\begin{abstract}
In this paper,  we develop a family of high order cut discontinuous Galerkin (DG) methods  for  hyperbolic conservation laws in one space dimension. The ghost penalty stabilization is used to stabilize the scheme for small cut elements. The analysis shows that our proposed methods have similar stability and accuracy properties as the standard DG methods on a regular mesh. We also prove that the cut DG method with piecewise constants in space is total variation diminishing (TVD). We use the strong  stability preserving Runge-Kutta method for time discretization and the time step is independent of the size of cut element.  Numerical examples demonstrate that the cut DG methods are high order accurate for smooth problems and perform well for discontinuous problems.

\end{abstract}

\section{Introduction}
In this paper, we will develop a high order cut discontinuous Galerkin (DG) method to solve one dimensional scalar hyperbolic conservation laws,
\begin{equation}\label{eq:linear:1d}
\left\{\begin{array}{l}
{u_t+(f(u))_x=0, \quad x\in\Omega=(x_l,x_r),\quad t>0}, \\
{u(x,0) = u_0(x),}
\end{array}\right.
\end{equation}
with inflow or periodic boundary condition.   The DG method was first introduced in 1973 by Reed and Hill \cite{Reed1973Triangular}.
A breakthrough  was made by Cockburn et al. in \cite{ShuDG2,ShuDG3,ShuDG4,ShuDG5} to solve hyperbolic conservation laws, coupled with  Runge-Kutta method \cite{shu1988efficient} for time discretization and total variation bounded (TVB) nonlinear limiters \cite{shu1987tvb} to achieve non-oscillatory properties in the presence of strong shocks.
 DG methods have low dispersion and dissipation errors for hyperbolic problems \cite{hesthaven2007nodal} and have broad applications in many areas.  We refer to \cite{li2005discontinuous,hesthaven2007nodal,riviere2008discontinuous} for more details.

Most applications based on DG methods are defined on fitted meshes.
To achieve the full potential accuracy of a DG method the mesh quality needs to be high. This requirement  can be problematic for problems posed on complicated domains and implies remeshing for problems on  moving domains or moving interfaces.  Typically the resulting mesh will contain small elements, which
for time-dependent problems may lead to very severe time-step restrictions. Small elements or cells cause severe problems also for continuous element methods and finite volume methods. Examples of efforts to overcome these difficulties are the finite volume based h-box method, Helzel et el \cite{berger2003h,berger2012simplified} and the recent stabilized DG method by Engwer et. al \cite{engwer2019stabilized}, where transport along the characteristics is explicitly taken  into account.
Many unfitted methods have also been developed, as for example
 the extended finite element method \cite{fries2010extended}, immersed boundary methods \cite{mittal2005immersed}, and unfitted finite element methods \cite{bordas2018geometrically}.

In  unfitted finite element methods, the physical domain is immersed in a background mesh. A possible approach is to solve for all degrees of  freedom corresponding to the smallest set of elements covering the physical domain. The weak forms defining the numerical scheme are defined on the physical domain. Thus, for each element integration will only be done over the part that  intersects the physical domain.  If the intersection is very small, this will result in ill-conditioning of the mass and stiffness matrices, and a very severe time-step restriction. There are two common approaches to overcome these problems. One is the cell merging (or agglomeration) \cite{Johansson2013,kummer2017extended,modisette2010toward,muller2017high,Qin2013,schoeder2018high}, where the small cut part is absorbed into a neighbour element by  extending the basis functions of the neighbour element. The other common approach is to add ghost penalty stabilization terms to the weak form \cite{burman2012fictitious,hansbo2014cut,massing2014stabilized,sticko2016stabilized, sticko2016higher}.
Some efforts have  been made to develop DG methods with ghost stabilization. G\"{u}rkan and Massing consider elliptic  problems \cite{GURKAN2019466} and stationary hyperbolic equations \cite{grkan2018stabilized}.

In this paper, we develop a family of cut DG methods with ghost penalty stabilization for first order hyperbolic problems. We consider linear and nonlinear scalar equations in one space dimension.  Our scheme is based on the standard DG methods, which saves considerable implementational effort.  We add stabilization terms on the interior  interfaces between cut elements and their neighbours. We will analyse the stability and accuracy of the resulting cut DG method, and discuss how the stabilization terms effect the condition numbers and eigenvalues of the involved matrices.  For the linear advection equation we prove an accuracy result, which is a half order lower compared to the accuracy of the standard DG method. However, we will report numerical observations showing optimal accuracy, both for linear and nonlinear problems. We also prove the total variation diminishing (TVD) property of the cut DG scheme with piecewise constants in space and explicit forward Euler time discretization. We can not prove the TVD property for the mean value (TVDM) for the proposed method based on high order polynomial spaces. However, we have observed numerically that the total variation is bounded for high order polynomials and explicit time stepping  for problems with smooth solutions, and for discontinuous solutions when we apply the TVD \textit{minmod} limiter. 
On coarse meshes with one small cut element and meshes with more small cut elements, some oscillations are triggered when a discontinuity passes a cut element and its neighbours. By applying a more robust limiting near cut elements, which explicitly lowers the polynomial degree when non-smoothness is detected, these oscillations can be avoided.

This paper is organized as follows. We will introduce notation and definitions of spaces and projections in Section \ref{sec:notapro}. The proposed cut DG method is given in Section \ref{sec:method}, and we discuss how the stabilization terms effect the condition number of the mass matrix and the eigenvalues of spatial discretization matrices. The stability and \emph{a priori} error estimate will be analysed in Section \ref{sec:stable}. Section \ref{sec:Numerical} contains some numerical results. The paper is concluded with a summary and a brief discussion of how the proposed ideas can be extended to hyperbolic systems and  higher dimensions.

\section{Discrete spaces  and projections}
\label{sec:notapro}
Let the computational domain be $\mathcal{T}=[x_{L},x_{R}]$ partitioned  by $x_{L}=x_{\frac{1}{2}}<x_{\frac{3}{2}}<\cdots<x_{N+\frac{1}{2}}=x_{R}$, which defines  the background mesh in which the physical domain $\Omega=[x_l,x_r]$ is immersed. Let $I_{j}=[x_{j-\frac{1}{2}},x_{j+\frac{1}{2}}]$ denote an element and $\mathcal{E}$ denote the set containing the edges in the background mesh. For simplicity, we assume all $I_j$'s have length $h$.  Define the cut mesh as
 \begin{align}
 &\Th =\{ I_j \cap\Omega\neq\emptyset: j=1,\cdots, N\},\Tr=\{I_j\in\Th|I_j\cap \Gamma \neq \emptyset\},\\
& \Fh=\{F \in \mathcal{E}\cap\Omega\}, \Fr=\{F\in\Fh=I_j\cap I_{k}|I_{j}\in\Tr, d(I_j,\Gamma)<0.5h, j\neq k\}.
 \end{align}
Here, $\Gamma$ is the boundary point $x_l$ or  one or more interface points $x_p$, each of which cuts one background element into two cut elements and $d(I_j,\Gamma)$ denotes the cut size of the element $I_j$ cut by $\Gamma$. We will consider two settings.
In the first, the boundary is immersed in the background mesh and the physical boundary point $x_l$ cuts the left most element $I_1$ with the cut size $\alpha h$, $\alpha\in[0,1]$.  We assume $\alpha\ll 1$, since this will be the difficult case. We have 
$\Tr=I_1, \Fr=x_{\frac{3}{2}}$.
\begin{figure}[tbhp]
\centering
\begin{tikzpicture}[xscale=4]
\draw[dotted][draw=black, very thick] (0,0) -- (.25,0);
\draw[-][draw=red, very thick] (0.25,0) -- (.4,0);
\node[red][above] at (0.32,0) {${\alpha h}$};
\draw[-][draw=black, very thick] (.4,0) -- (0.8,0);
\draw[dotted][draw=black, very thick] (0.8,0) -- (1.2,0);
\draw[-][draw=black, very thick] (1.2,0) -- (1.6,0);
\draw[dotted][draw=black, very thick] (1.6,0) -- (2.0,0);
\draw[-][draw=black, very thick] (2.0,0) -- (2.4,0);
\draw [thick] (0,-.5) node[below]{$x_L=x_{\frac{1}{2}}$} -- (0,0.2);
\draw[red][thick] (0.25,-.2)-- (0.25,.5);
\node[red][above] at (0.25,.5) {$x_l$};
\draw [thick] (0.4,-.1) node[below]{$x_{\frac{3}{2}}$} -- (0.4,0.2);
\draw [thick] (0.8,-.1) node[below]{$x_{\frac{5}{2}}$} -- (0.8,0.2);
\draw [thick] (1.2,-.1) node[below]{$x_{j-\frac{1}{2}}$} -- (1.2,0.2);
\draw [thick] (1.6,-.1) node[below]{$x_{j+\frac{1}{2}}$} -- (1.6,0.2);
\draw [thick] (2.0,-.1) node[below]{$x_{N-\frac{1}{2}}$} -- (2.0,0.2);
\draw [thick] (2.4,-.5) node[below]{$x_R=x_{N+\frac{1}{2}}$} -- (2.4,0.2);
\draw [red][thick] (2.4,-.2) -- (2.4,0.5);
\node [red][above] at (2.4,0.5) {$x_r$};
\draw [red][thick, <->] (0.25,.5) -- (2.4,0.5);
\node [red][above] at (1.2,0.4){physical domain $\Omega$};
\draw [black][thick, <->] (0.,-1) -- (2.4,-1);
\node [black][below] at (1.0,-1){computational domain $\mathcal{T}$};
\end{tikzpicture}
\caption{Discretization of the physical domain $\Omega$ and the computational domain $\mathcal{T}$.}\label{mesh_discretization}
\end{figure}
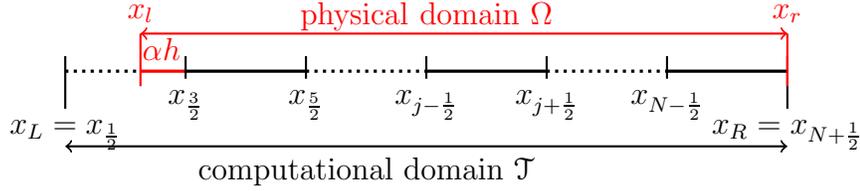

 In the second setting we consider problems with one or more interior interfaces, resulting in cut elements located in the interior part of the domain, see Figure \ref{fig:cutmiddle1}.  For each interface $x_p$ there are two cut elements with cut size  $\alpha h$ and $(1-\alpha)h$, respectively, and the computational domain $\mathcal{T}$ is equal to  $\Omega$. In this setting, $\Tr=I_J,\Fr=\{x_{J-\frac{1}{2}}\}$.
\begin{figure}[tbhp]
\centering
\begin{tikzpicture}[xscale=4]
\draw[-][draw=black, very thick] (0,0) -- (.4,0);
\draw[dotted][draw=black, very thick] (0.4,0) -- (0.8,0);
\draw[-][draw=black, very thick] (0.8,0) -- (1.2,0);
\draw[-][draw=red, very thick] (1.2,0) -- (1.35,0);
\node[red][above] at (1.3,0.1) {${\alpha h}$};
\draw[-][draw=red, very thick] (1.35,0) -- (1.6,0);
\node[red][above] at (1.5,0.3) {${(1-\alpha)h}$};
\draw[-][draw=black, very thick] (1.6,0) -- (2.0,0);
\draw[dotted][draw=black, very thick] (2.0,0) -- (2.4,0);
\draw[-][draw=black, very thick] (2.4,0) -- (2.8,0);
\draw [thick] (0,-.5) node[below]{$x_L=x_{\frac{1}{2}}$} -- (0,0.2);
\draw[red][thick] (0.0,-.2)-- (0.0,.5);
\node[red][above] at (0.0,.5) {$x_l$};
\draw [thick] (.4,-.1) node[below]{$x_{\frac{3}{2}}$} -- (0.4,0.2);
\draw [thick] (0.8,-.1) node[below]{$x_{J-\frac{3}{2}}$} -- (0.8,0.2);
\draw [thick] (1.2,-.1) node[below]{$x_{J-\frac{1}{2}}$} -- (1.2,0.2);
\draw [blue][thick] (1.35,-0.1) node[below]{$x_{p}$} -- (1.35,0.2);
\draw [thick] (1.6,-.1) node[below]{$x_{J+\frac{1}{2}}$} -- (1.6,0.2);
\draw [thick] (2.0,-.1) node[below]{$x_{J+\frac{3}{2}}$} -- (2.0,0.2);
\draw [thick] (2.4,-.1) node[below]{$x_{N-\frac{1}{2}}$} -- (2.4,0.2);
\draw [thick] (2.8,-.5) node[below]{$x_R=x_{N+\frac{1}{2}}$} -- (2.8,0.2);
\draw [red][thick] (2.8,-.2) -- (2.8,0.5);
\node [red][above] at (2.8,0.5) {$x_r$};
\end{tikzpicture}
\caption{Discretization of the physical domain $\Omega$ with the mesh partition having equal size $h=(x_r-x_l)/N$ and the middle element is split into two elements of length $\alpha h$ and $(1-\alpha)h$.}\label{fig:cutmiddle1}
\end{figure}
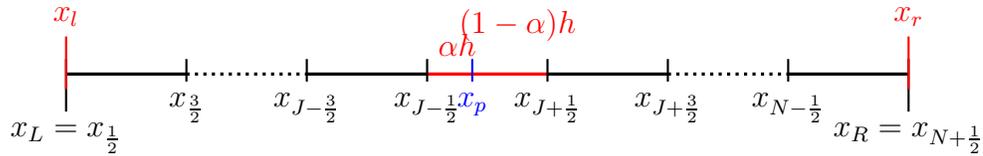
The background element $I_J$ includes  $[x_{J-\frac{1}{2}}, x_{p}]$ and $[x_p,x_{J+\frac{1}{2}}]$. Thus, $I_J$ will be used twice in the weak formulation as  background element for both cut elements.  
Note that for the periodic problem, a boundary point can equivalently be taken as an interior point, and analyzing the mesh setting in Figure \ref{mesh_discretization} suffices.

As for the standard DG method, we use the piecewise polynomial space
\begin{equation}
\V=\{{v_h} : {v_h}|_{I_j} \in P^r(I_j), \forall I_j \in \mathcal{T}_h\},
\end{equation}
where $P^{r}(I_j)$ is the space of polynomials with  degree at most $r$ in $I_j$.   We note that the  space $\V$ and basis functions of space $\V$ are defined on the background mesh. As usual, we define the average and the jump of a function $v$ at $x$ as
\begin{align}
 \{v\}=\frac{1}{2}(v^++v^-), \quad [v]=(v^+-v^-).
 \end{align}
For any $v\in\V$, $v^+=\lim\limits_{\epsilon\to 0^+}v(x+\epsilon)$ and $v^-=\lim\limits_{\epsilon\to 0^+}v(x-\epsilon) $ denote the limit values of $v$ at $x$ from right and left.
For square integrable functions on a given domain $K$, the  inner product and the $L^2$ norm are denoted
\begin{align}\label{L2}
(v,w)_{K}:=\int_{K} vwdx, \quad \|v\|_K:=\sqrt{(v,v)_{K}}, \quad \forall v, w\in L^2(K).
\end{align}

\section{The DG scheme}
\label{sec:method}
In this section, we will present the cut DG method  for the scalar equation \eqref{eq:linear:1d} with periodic boundary condition, on the mesh with one cut element on the left boundary, as in Figure \ref{mesh_discretization}. We will show how  we  stabilize and investigate how this effects the condition number and the eigenvalues of the  matrices from the  space discretization for the linear advection equation.

\subsection{The cut DG scheme without stabilization}
 In this subsection, we use  a methods of lines approach. We will discretize in space and give the semi-discrete DG method for the scalar hyperbolic equation \eqref{eq:linear:1d}. We multiply the equation \eqref{eq:linear:1d} by the test function $v_h$ and integrate it by parts. Then we get the semi-discrete DG method which is to  look for $u_h(\cdot,t)$ $\in \V$ such that for any $v_h$ $\in \V$, and for all $I_j\in\Th$,
\begin{align}
\label{scheme:linear:1d}
\left( u_{h,t},v_{h} \right)\Ij +\widehat{f}_{jr}v_{h,jr}^--\widehat{f}_{jl}v_{h,jl}^+- (f(u_h),(v_{h})_x)\Ij&=0.
\end{align}
Here $\widehat{f}_*=\widehat{f}\left(u_h^-(x_*,t), u_{h}^+(x_*,t))\right)$ is numerical flux function to approximate flux $f(u)$ at the point $x_*$, and  $x_{jl}$ and $x_{jr}$ denote the end points of $I_j\cap\Omega$ representing $x_l$ or $x_{j\pm\frac{1}{2}}$. As in the standard DG method \cite{shu2009discontinuous}, we use monotone fluxes  which satisfies consistency: $\widehat{f}(u,u)=f(u)$, continuity: $\widehat{f}\left(u^{-}, u^{+}\right)$ is at least Lipschitz continuous with respect to both arguments, and monotonicity $\widehat{f}(\uparrow,\downarrow)$: $\widehat{f}\left(u^{-}, u^{+}\right)$ is a non-decreasing function of its first argument  and a non-increasing function of its second argument. 
We sum the cut DG scheme \eqref{scheme:linear:1d} over $j$ and write it as
\begin{align}\label{scheme:cutDG0}
&\left({u}_{h,t}, v_h\right)_\Omega+a\left(u_{h}, v_h\right)=0, \quad \forall v_h \in \V,\\
&a(u,v)=\widehat{f}_{N+\frac{1}{2}} v^{-}_{N+\frac{1}{2}}-\widehat{f}_lv_l-\sum_{j=2}^{N}\widehat{f}_{j-\frac{1}{2}}[v]_{j-\frac{1}{2}}-\sum_{j=1}^{N}(f(u),v_x)_{I_j\cap\Omega}.
\label{def:auv}
\end{align}

Next, we will consider the matrix form of the cut DG scheme \eqref{scheme:cutDG0} for the periodic linear advection equation with $f(u)=\beta u$ and $\beta>0$ in \eqref{eq:linear:1d}
\begin{align}
&u_t+(\beta u)_x=0,\quad x\in(x_l,x_r),\quad t>0.
\label{eq:adv:1d}
\end{align}
Let $\{\phi^k_j(x)\}_{k=0}^r$ be the basis of $P^r$ on the element $I_j$, and express  $u_h(\cdot,t)$ $\in \V$ as
\begin{align}\label{eq:def:uh}
u_h|_{I_j}=\sum_{k=0}^{r}u^k_j(t)\phi^k_j(x),
\end{align}
with $u_j^k(t)$ being the unknown time-dependent coefficients of numerical solution $u_h$. 
Our implementation uses the orthogonal Legendre  polynomials $1, \xi_{j}$, $ \xi_{j}^{2}-1/3$, $\cdots,$ with $\xi_{j}=\frac{x-x_{j}}{h/ 2}$ as the basis functions in each element $I_j$ and $x_j$ denotes the centre point of $I_j$. Thus, the local mass matrices on regular elements are diagonal, while they will be full on the cut elements because of the integration over only a part of the element. 
Introduce $u_h$ \eqref{eq:def:uh} into the cut DG scheme \eqref{scheme:cutDG0} with the upwind flux $\widehat{f}(u_h^-,u_h^+)=\beta u_h^-$ for equation \eqref{eq:adv:1d} with periodic boundary condition leads to the matrix form of the semi-discrete problem,
\begin{equation}\label{eq:def:matrixform}
\mathcal{M}U_t=SU ,\text{or } U_t=\mathcal{M}^{-1}SU.
\end{equation}
Here $\mathcal{M}$ is the block-diagonal mass matrix,  $S$ is the stiffness matrix, and $U$ is the coefficients vector of $u_h$, which will be a smooth function of $t$. The temporal stability of the semi-discrete scheme \eqref{scheme:cutDG0} depends on the eigenvalues of ${\mathcal{M}}^{-1}{S}$. 
 For \eqref{eq:def:matrixform} to be stable the eigenvalues of ${\mathcal{M}}^{-1}{S}$ need to be in the negative half-plane.
 If this is satisfied,  eigenvalues with large absolute values will set the time-step limit for explicit time-stepping methods. Also interesting is the mass matrix condition number, since ill-conditioning here may cause difficulties in time-stepping.

 In Table \ref{table:withoutstabilization}, we show results for the domain $[0,2]$ with a background mesh of $8$ elements, where  one cut element is located on the left boundary. We use  $\beta=1$ and consider $\alpha=10^{-2},10^{-10}$.
Note that for the higher polynomial orders eigenvalues with positive real part appear as alpha goes to zero, which indicates instability. For lower orders  eigenvalues of large magnitude appear, which cause severe explicit time-stepping limits.
The results also indicate a sharply increasing  condition number of the  mass matrix,  with the degree of the polynomials space, and as  the size of the cut goes to zero. To overcome these problems, we will include stabilization.

\begin{table}[tbhp]
\caption{\label{table:withoutstabilization} Condition number of the mass matrix ${\mathcal{M}}$  and the maximal absolute value and maximal real part of the eigenvalue $v_i$ of the spacial operator (${\mathcal{M}}^{-1}S$) in the cut DG scheme without stabilization.  
}
\begin{scriptsize}
\begin{tabular}{|c|ccc|ccc|}
 \hline	
 &\multicolumn{3}{c|}{$\alpha={10^{-2}}$} &\multicolumn{3}{c|}{$\alpha={10^{-10}}$} \\\hline	
 degree	&	$\mathcal{K(M)}$ & $\max(|v_i|)$	&	$\max(Re(v_i))$		&$\mathcal{K(M)}$  &	$\max(|v_i|)$	&	$\max(Re(v_i))$	\\
\hline	
 $P^0$	&	1.00E+02	&	3.51E+02	&	-3.51E+00	&	1.00E+10	&	3.50E+10	&	-3.50E+00	\\
$P^1$	&	5.94E+06	&	8.59E+02	&	-6.97E+00	&	1.31E+26	&	3.50E+10	&	-6.96E+00	\\
$P^2$	&	7.48E+11	&	1.42E+03	&	-9.25E+00	&	2.65E+26	&	5.72E+10	&	1.24E+10	\\
$P^3$	&	1.15E+17	&	2.05E+03	&	-1.10E+01	&	2.34E+27	&	5.40E+12	&	5.40E+12	\\
\hline		
\end{tabular}
\end{scriptsize}
\end{table}

\subsection{The cut DG method with ghost penalty stabilization}
In this subsection the starting point is a scaled version of the ghost penalty suggested in \cite{massing2014stabilized},
\begin{equation}\label{stable0}
j_1(u, v)=\sum_{F \in \Fr} \sum_{k=1}^{r} w_kh^{2k+1} \left[\partial^{k} u\right]_F\left[\partial^{k} v\right]_{F}.	
\end{equation}
This stabilization  is also suggested for a wave propagation problem in \cite{sticko2016higher}.
In this work we propose to use
\begin{equation}\label{stable1}
J_s(u, v)=\sum_{F \in \Fr} \sum_{k=0}^{r}  w_{k}h^{2k+s}\left[\partial^{k} u\right]_F\left[\partial^{k} v\right]_{F}.	
\end{equation}
to stabilize both mass and stiffness matrices. Compared to $j_1(u,v)$ \eqref{stable0} we have included the jump of $u$ to ensure that all terms are stabilized.
With the stabilization $J_1(u_t,v)$ and $J_0(u,v)$, the semi-discrete cut DG scheme  has the following form:  look for $u_h(\cdot,t)$ $\in \V$ such that for $\forall v_h\in\V$,
\begin{align}\label{scheme:cutDG2}
(u_{h,t},v_h)_\Omega+\gamma_M J_1(u_{h,t},v_h)+a(u_h,v_h)+\gamma_A J_0(u_h,v_h)=0.
\end{align}
Here, $a(u,v)$ are defined in \eqref{def:auv} and $\rM,\rA$ are positive constants. With $v_h=1$ in \eqref{scheme:cutDG2}, we get
\begin{equation}
\frac{d}{dt}\int_\Omega u_hdx +\widehat{f}_{N+\frac{1}{2}}-\widehat{f}_l=0.
\end{equation}
Thus,  the cut DG scheme with stabilization $J_1(u_{h,t},v_h)$ and $J_0(u_h,v_h)$ is globally conservative. We note that it is locally conservative for all elements which do not have a cut element as a neighbour, and for on all patches of elements that are connected by faces where the stabilization is applied.

Next, we study the eigenvalues and condition numbers of the  matrices resulting from applying  the stabilized cut DG scheme \eqref{scheme:cutDG2}  to the linear advection equation \eqref{eq:adv:1d}  with parameter values $\rM=0.25,\rA=0.75$ and $\omega_k=\frac{1}{(2k+1)k!^2}$ as  in \cite{sticko2016higher}. We write the scheme \eqref{scheme:cutDG2} with the periodic boundary condition in the matrix form as
\begin{align}\label{scheme2}
\widetilde{\mathcal{M}}U_t=\widetilde{\mathcal{S}}U.
\end{align}
Here, $\widetilde{\mathcal{M}}$ is stabilized mass matrix with 
$\widetilde{\mathcal{M}}U_t=(u_{h,t},v_h)_\Omega+\gamma_M J_1(u_{h,t},v_h)$ and $\widetilde{\mathcal{S}}$ is stabilized stiffness matrix with 
$\widetilde{\mathcal{S}}U=-a(u_h,v_h)-\gamma_A J_0(u_h,v_h)$. 
In Table \ref{table:stablemassstiff} the condition number of $\tilde{\mathcal{M}}$, the maximal absolute value and the maximal real part of the eigenvalues of the spatial operator ($\widetilde{\mathcal{M}}^{-1}\widetilde{\mathcal{S}}$) are listed.   Comparing the results in Table \ref{table:stablemassstiff} and Table \ref{table:withoutstabilization}, the condition number of $\widetilde{\mathcal{M}}$ is much smaller than for the original mass matrix. Besides, the maximal absolute eigenvalue and its real part from the stabilized scheme are almost the same as  for the standard DG method on a uniform mesh in Table \ref{table:uniform}.  We therefore expect similar time-step restrictions for the stabilized cut DG method as for the standard  DG method. We have also tested the scheme with more small cut elements in the interior and get very similar results as in Table \ref{table:stablemassstiff}.
\begin{table}[htbp]
\caption{\label{table:stablemassstiff} {Condition number of mass matrix $\widetilde{\mathcal{M}}$  and maximal absolute value and maximal real part of the eigenvalue $v_i$ of the spatial operator ($\widetilde{\mathcal{M}}^{-1}\tilde{S}$) in the cut DG scheme with stabilization.
}}
\begin{scriptsize}
\begin{tabular}{|c|ccc|ccc|}
 \hline
 &\multicolumn{3}{c|}{$\alpha={10^{-2}}$} &\multicolumn{3}{c|}{$\alpha={10^{-10}}$} \\\hline	
 degree	&	$\mathcal{K(\widetilde{\mathcal{M}})}$ & $\max(|v_i|)$	&	$\max(Re(v_i))$	&$\mathcal{K(\widetilde{\mathcal{M}})}$  &	$\max(|v_i|)$	&	$\max(Re(v_i))$	\\
\hline	
$P^0$	&	6.53E+00	&	2.34E+01	&	3.47E-15		&	6.85E+00	&	2.45E+01	&	-8.39E-17	\\
$P^1$	&	4.79E+01	&	2.22E+01	&	5.93E-16		&	5.07E+01	&	2.45E+01	&	-2.56E-15	\\
$P^2$	&	3.77E+03	&	4.08E+01	&	1.12E-15		&	4.04E+03	&	4.11E+01	&	5.33E-16	\\
$P^3$	&	8.58E+05	&	6.69E+01	&	5.07E-15		&	9.39E+05	&	6.70E+01	&	-2.53E-16	\\
$P^4$	&	1.93E+08	&	9.65E+01	&	4.34E-15		&	2.16E+08	&	9.67E+01	&	-5.60E-16	\\
 \hline
\end{tabular}
\end{scriptsize}
\end{table}

In Figure \ref{figure:eigenvalue},  we plot the eigenvalues of  $\Delta t\widetilde{\mathcal{M}}^{-1}\widetilde{\mathcal{S}}$ from  the stabilized cut DG scheme \eqref{scheme:cutDG2} for different polynomials and varying  cut size with $\alpha=10^{-2},10^{-10}$. In our case the physical domain is $[0,2]$, and in the cut cases the background mesh consists of  8 elements,  of which two cut elements $[1,1+\alpha h], [1+\alpha h,1+h]$ are included as  in Figure \ref{fig:cutmiddle1}. The time steps are taken to be $\Delta t=\lambda h$ with $\lambda=0.3,0.2,0.1,0.1$ for $r=1,2,3,4\in\V$ polynomials, as for the standard DG method. From the results, we can observe that the eigenvalues are all located in the stable region of fourth order Runge-Kutta method. Thus, our proposed cut DG methods can use a similar time step as the standard DG method even with a very small cut cell.
\begin{figure}[tbhp]
\begin{center}
\includegraphics[width=2.2in]{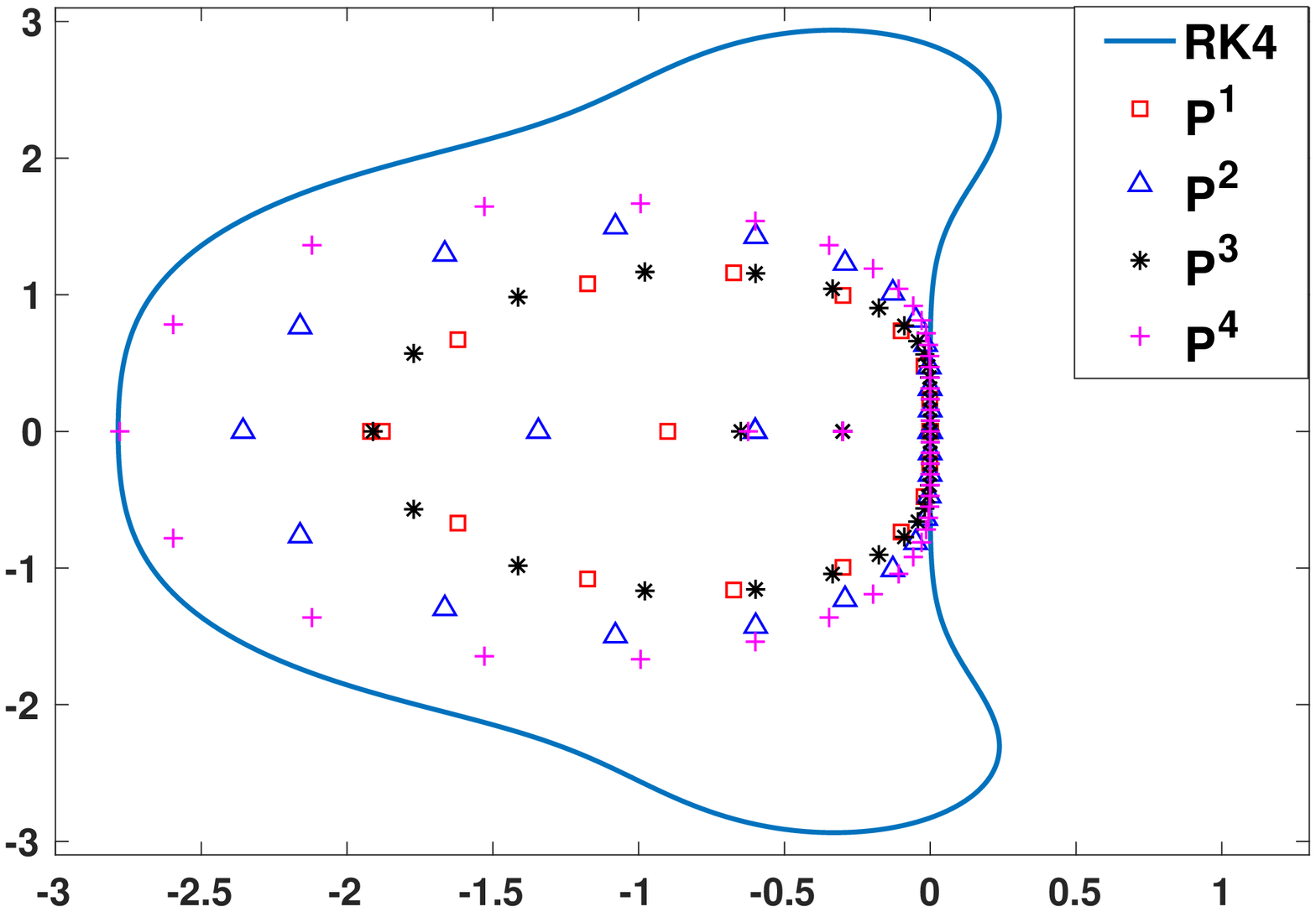}
\includegraphics[width=2.2in]{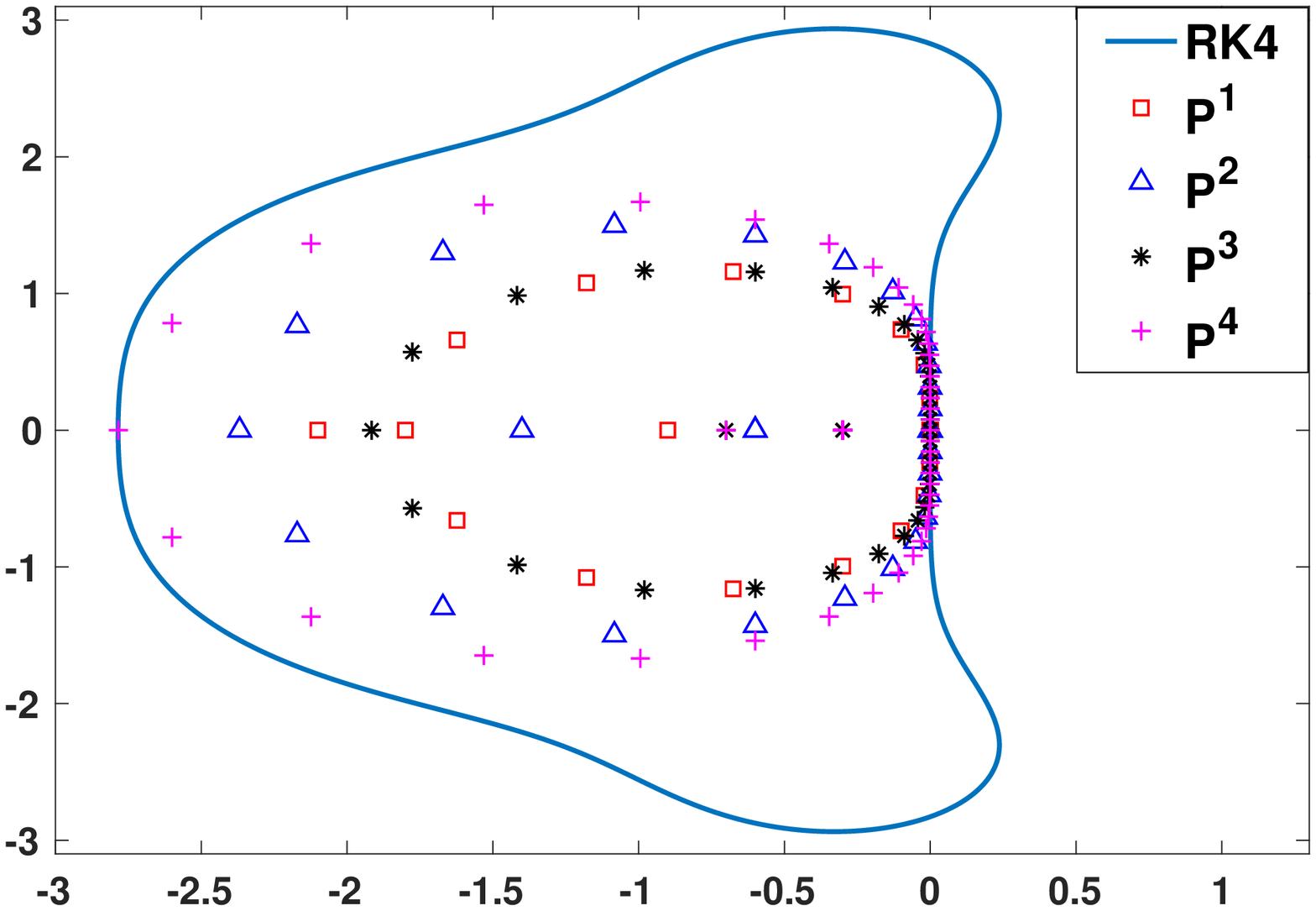}
\caption{The eigenvalue distributions $v$ of $\Delta t\widetilde{\mathcal{M}}^{-1}\widetilde{\mathcal{S}}$ from  the cut DG scheme with different polynomial spaces $P^r$ and $\alpha=10^{-2}$ (left), $10^{-10}$ (right) on the mesh setting in Figure \ref{fig:cutmiddle1}. }
\label{figure:eigenvalue}
\end{center}
\end{figure}

\subsection{Condition number of the mass matrix with stabilization}
Without stabilization the mass matrix is a block diagonal matrix with $(r+1)\times(r+1)$ local stiffness matrices as diagonal blocks. For regular elements these blocks are diagonal. With the stabilization term, the
mass matrix couples the degrees of freedom in cut elements  with those of their neighbours. This coupling gives rise to a $2k(r+1)\times2k(r+1)$ matrix with $k$ being the number of  stabilized cut elements. 
Note that the global mass matrix $\widetilde{\mathcal{M}}$ is still block-diagonal but not element by element.
Following the analysis in \cite{sticko2016higher}, we can  estimate the condition number of the mass matrix. 
\begin{lem}\label{lem:conditionnumber}
The condition number of the modified mass matrix with stabilization $J_1(u_t,v)$  with $\rM>0$ has an  upper bound
\begin{equation}
\kappa(\widetilde{\mathcal{M}}) \leq C_{M}\kappa\left(\mathcal{M}\right).
\end{equation}
Here, $\mathcal{M}$ is the mass matrix for the standard DG scheme on a uniform mesh. For fixed $\rM > 0$ and fixed weights $\omega_k$ the constant $C_M$ depends only on the polynomial order, and not on  $h$ or $\alpha$. It will grow unboundedly if $\rM$ approaches zero.
 \end{lem}
For details and proof of Lemma \ref{lem:conditionnumber} we refer to the paper \cite{sticko2016higher}.

\section{Analysis of the cut DG method with stabilization}
\label{sec:stable}
In this section we  analyze the stabilized cut DG scheme \eqref{scheme:cutDG2} with respect to stability and accuracy.

\subsection{$L^2$ stability}
We will use the entropy inequality in \cite{jiang1994cell} to prove that the stabilized cut DG method \eqref{scheme:cutDG2} for  periodic scalar hyperbolic problems is $L^2$ stable.
\begin{thm}
With the parameters $\gamma_M\geq0, \gamma_A\geq0$ and flux $\widehat{f}$ satisfying consistency, monotonicity and continuity, the semi-discrete cut DG scheme \eqref{scheme:cutDG2} for the periodic scalar hyperbolic equation \eqref{eq:linear:1d} satisfies
 \begin{align*}
||u_h(\cdot,t)||^2_\Omega+\rM J_1(u_h(\cdot,t),u_h(\cdot,t))\leq ||u_h(\cdot,0)||^2_\Omega+\rM J_1(u_h(\cdot,0),u_h(\cdot,0)).
\end{align*}
\end{thm}
\begin{proof}
Taking $v_h=u_h(\cdot,t)$  in the scheme \eqref{scheme:cutDG2} we get
 \begin{equation}
 (u_{h,t},u_h)_\Omega + \gamma_M J_1(u_{h,t},u_h) + a(u_h,u_h) +\gamma_A J_0(u_h,u_h)=0.
 \end{equation}
That is
\begin{align}
0=&\ino u_{h,t}u_{h} dx+\rM J_1(u_{h,t},u_h) +\rA J_0(u_h,u_h)\notag\\
&+\sum_j\left(-\inj f\left(u_{h}\right)\left(u_{h}\right)_{x} d x+\widehat{f}_{jr} u_{h,jr}^--\widehat{f}_{jl} u_{h,jl}^+
\right).
\label{eq:stable1}
\end{align}
We define $G_j=-\inj f\left(u_{h}\right)\left(u_{h}\right)_{x} d x+\widehat{f}_{jr} u_{h,jr}^--\widehat{f}_{jl} u_{h,jl}^+$ and $\tilde{F}(u)=\int^{u} f(u) du$.  Then, we have
\begin{align*}
G_j
=&\left(-\tilde{F}\left(u_{h,jr}^-\right)+\widehat{f}_{jr} u_{h,jr}^-\right)-\left(-\tilde{F}\left(u_{h,jl}^-\right)+\widehat{f}_{jl} u_{h,jl}^- \right)+\Theta_{jl}.
\end{align*}
Here,
$$
\Theta_{jl}=-\tilde{F}\left(u_{h,jl}^-\right)+\widehat{f}_{jl} u_{h,jl}^-+\tilde{F}\left(u_{h,jl}^+\right)-\widehat{f}_{jl} u_{h,jl}^+.
$$
By the mean value theorem it follows that
\begin{align}
\Theta&=-\tilde{F}\left(u_{h}^{-}\right)+\widehat{f} u_{h}^{-}+\tilde{F}\left(u_{h}^{+}\right)-\widehat{f} u_{h}^{+}=\left(u_{h}^{+}-u_{h}^{-}\right)\left(\tilde{F}^{\prime}(\xi)-\widehat{f}(u_h^-,u_h^+)\right)\\
&=(u_h^+-u_h^-)(f(\xi)-\widehat{f}(u_h^-,u_h^+)) \geq 0,\notag
\end{align}
where $\xi$ is a value between $u_h^{-}$ and $u_h^{+}$, and the monotonicity $\widehat{f}(\uparrow,\downarrow)$ as well as consistency $f(\xi)=\widehat{f}(\xi,\xi)$ of flux function $\widehat{f}$ are used.  Thus, the equation \eqref{eq:stable1} becomes
\begin{align}
0=&\ino u_{h,t}u_{h} dx+\rM J_1(u_{h,t},u_h) +\rA J_0(u_h,u_h)+\sum_jG_j\notag\\
=&\ino u_{h,t}u_{h} dx+\rM J_1(u_{h,t},u_h) +\rA J_0(u_h,u_h)+\sum_j \Theta_{jl}.
\end{align}
Here, $\sum\limits_j \left(\left(-\tilde{F}\left(u_{h,jr}^-\right)+\widehat{f}_{jr} u_{h,jr}^-\right)-\left(-\tilde{F}\left(u_{h,jl}^-\right)+\widehat{f}_{jl} u_{h,jl}^- \right)\right)=0$ follows from the periodic boundary condition.  
With $\rM\geq 0, \rA\geq 0$ in the stabilization we have
\begin{align}
\frac{1}{2}\frac{d}{dt}\left(\ino u_h^2dx+\rM J_1(u_h,u_h)\right)=\ino u_{h,t} u_{h} dx+\rM J_1(u_{h,t},u_h) \leq 0.
\end{align}
Integrating this relation in time completes the proof.
\end{proof}
Further, with $\rM>0$ Theorem 4.1 can be combined with Lemma 3.1 to yield a corresponding estimate for the coefficient vector $|U(t)|=\sqrt{U^TU}$, which is independent of the cut size $\alpha$. We point out that the semi-discrete cut DG scheme \eqref{scheme:linear:1d} without stabilization also satisfies  $L^2$ stability, and that the terms $J_1(u_t,v)$ and $J_0(u,v)$ do not destroy the stability. Comparing with standard DG method, we can not have the cell entropy inequality in each cell because  the stabilization terms combine the cut cells and their neighbours.

\subsection{\emph{a priori} error estimate}
In this subsection, we will derive an  \emph{a priori} error estimate for the  semi-discrete stabilized cut DG scheme \eqref{scheme:cutDG2} applied to  the linear advection equation \eqref{eq:adv:1d}. 

First, we give some properties of the unfitted $L^2$ projection $\p$, which are stated in \cite{grkan2018stabilized}.
The projection $\p $ satisfies the following error estimates
\begin{align}
\label{eq:pi:property0}
\left\|v-\pi_{h}^{e} v\right\|_{\mathcal{T}_{h}, s} &\lesssim h^{r-s}\|v\|_{r, \Omega}, \qquad 0 \leqslant s \leqslant r,\\
\label{eq:pi:property1}
\left\|v-\pi_{h}^{e} v\right\|_{\mathcal{F}_{h}, s} &\lesssim h^{r-s-1 / 2}\|v\|_{r, \Omega}, \quad 0 \leqslant s \leqslant r-1 / 2,\\
\label{eq:pi:property2}
\left\|v-\pi_{h}^{e} v\right\|_{\Gamma, s} &\lesssim h^{r-s-1 / 2}\|v\|_{r, \Omega}, \quad 0 \leqslant s \leqslant r-1 / 2.
\end{align}
Here, the $\lesssim-$ relation denotes $a \lesssim b \Leftrightarrow a \leq C b$ with $C$ being some constant that is independent of $h$ and $r$. Further, $||\cdot||_{m,\Omega}$ denotes the usual norm of Sobolev spaces $H^m(\Omega)$ and $||\cdot||^2_{\Th,s}=\sum_{K\in\Th}||\cdot||^2_{s,K}$.

We assume $u(\cdot,t)$ is the exact solution of problem \eqref{eq:adv:1d} and $u_h(\cdot,t)\in \V$ is the approximate solution satisfying the cut DG scheme \eqref{scheme:cutDG2}. Thus, we have
\begin{align}\label{eq:exact}
(u_t,v_h)_\Omega +\rM J_1(u_t,v_h) + a(u,v_h) + \rA J_0(u,v_h)=0, \;\forall v_h \in \V.
\end{align}
Subtracting  equation \eqref{scheme:cutDG2} from \eqref{eq:exact}  yields  for $\forall v_h \in \V$
\begin{align}\label{eq:erroreq0}
((u-u_h)_t,v_h)_\Omega + \rM J_1((u-u_{h})_t,v_h) + a(u-u_h,v_h) +  \rA J_0(u-u_h,v_h)=0.
\end{align}
Furthermore, we assume $u\in L^\infty\left((0,t);H^{r+2}(\Omega)\right)$ and $u_t\in L^\infty\left((0,t);H^{r+1}(\Omega)\right)$. Then, we have the following error estimate.
\begin{thm}
We assume $u$  is a sufficiently smooth exact solution of the linear advection equation \eqref{eq:adv:1d} with periodic boundary condition,  and $u_h(\cdot,t)\in \V$ \eqref{eq:def:uh} is the approximation  satisfying the stabilized cut DG scheme \eqref{scheme:cutDG2}. Then, we have the \textit{a priori} error estimate
\begin{align}\label{ieq:error1}
||u(\cdot,t)-u_h(\cdot,t)||_\Omega^2\leq Ch^{2r+1}.
\end{align}
Here, $C$ is a constant depending on final time $t$,  on $u$ and its derivatives, and on parameters $\rM,\rA, \omega_k$ in the stabilization terms. 
 In particular, it is independent of the size of the cut geometries.
\end{thm}
\begin{proof}
Define $\eta=u-\p u, \theta=\p u-u_h$ and we have $\theta \in\V$ which is a smooth function of $t$. Then, we can rewrite the error equation \eqref{eq:erroreq0} as
\begin{align}\label{eq:erroreq1}
&(\theta_t,v_h)_\Omega+\rM J_1(\theta_t,v_h) + a(\theta,v_h) + \rA J_0(\theta,v_h) \notag\\
&=-(\eta_t,v_h)_\Omega + \rM J_1((\p u)_t, v_h) - a(\eta,v_h) + \rA J_0(\p u,v_h), \; for \;\forall v_h \in \V.
\end{align} 
Taking the test function $v_h=\theta(\cdot,t)$ at a particular time $t$ in equation \eqref{eq:erroreq1}, we have
\begin{align*}
Lhs:&=(\theta_t,\theta)_\Omega+\rM J_1(\theta_t,\theta) + a(\theta,\theta) + \rA J_0(\theta,\theta)\notag\\
&=\frac{1}{2}\frac{d}{dt}\left(||\theta||^2_\Omega+\rM J_1(\theta,\theta)  \right)+ \rA J_0(\theta,\theta)+\frac{\beta}{2}[\theta]_l^2+\frac{\beta}{2}\sum_{j=2}^{N}[\theta]^2_{j-\frac{1}{2}},\\
Rhs:&=-(\eta_t,\theta)_\Omega + \rM J_1((\p u)_t, \theta) - a(\eta,\theta) + \rA J_0(\p u,\theta)\notag\\
&= -(\eta_t,\theta)_\Omega + \rM J_1((\p u)_t,\theta) +\rA J_0(\p u,\theta)\notag\\
&\; -\left(-\beta\widehat{\eta}_l[\theta]_l^2-\sum_{j=2}^{N}\beta\widehat{\eta}_{j-\frac{1}{2}}[\theta]_{j-\frac{1}{2}}-\sum_{j=1}^{N}(\beta\eta, \theta_x)_{I_j\cap\Omega}\right).
\end{align*}
In the above two equations, we used
$\widehat{\theta}_{N+\frac{1}{2}}=\widehat{\theta}_l$, $\widehat{\eta}_{N+\frac{1}{2}}=\widehat{\eta}_l$ and $[\theta]_l=\theta(x_l^+,t)-\theta^-_{N+\frac{1}{2}}$ based on the periodic boundary condition.
Using Cauchy-Schwarz inequality, we obtain
\begin{align*}
&Rhs\leq \frac{1}{2}\left(||\eta_t||_\Omega^2+||\theta||^2_\Omega+\rM J_1((\p u)_t, (\p u)_t) + \rM J_1(\theta,\theta) + \rA J_0(\p u, \p u)\right)\\
&+\rA J_0(\theta, \theta)+
\frac{\beta}{2}\left(\sum_{j=2}^{N}[\theta]^2_\jl+[\theta]_l^2+\sum_{j=2}^{N} \hat{\eta}_\jl^2+\widehat{\eta}^2_l+ h^{-2}||\eta||_\Omega^2+ h^{2}||\theta_x||^2_\Omega\right).\notag
\end{align*}
Using the inverse inequality  $h||\theta_x||_\Omega\leq C||\theta||_\Omega$ and combining the above equations in this subsection, we can get
\begin{align}\label{ieq:0}
&\frac{d}{dt}\left(||\theta||^2_\Omega+\rM J_1(\theta,\theta)  \right)\notag\\
&\leq ||\eta_t||_\Omega^2+||\theta||^2_\Omega+\rM J_1((\p u)_t, (\p u)_t) + \rM J_1(\theta,\theta) + \rA J_0(\p u, \p u)\notag\\
&\; 
+\beta\sum_{j=2}^{N} \hat{\eta}_\jl^2+\beta \widehat{\eta}^2_l
+h^{-2}\beta||\eta||_\Omega^2+h^2\beta||\theta_x||^2_\Omega\notag\\
&\; \leq C_1(||\theta||_\Omega^2+\rM J_1(\theta,\theta)) + I+II+III+IV+V,
\end{align}
where
\begin{align*}
I &= ||\eta_t||_\Omega^2,\; II= \beta h^{-2}||\eta||_\Omega^2,III=\rM J_1((\p u)_t, (\p u)_t),\notag\\
IV&=\rA J_0(\p u, \p u), V={\beta}\sum_{j=2}^{N} \hat{\eta}_\jl^2+{\beta}\hat{\eta}^2_l.
\end{align*}
Here, $\eta_t=(u-\p u)_t=u_t-\p u_t$. Thus, based on the properties of $\p $ in \eqref{eq:pi:property0} and assuming $u\in L^\infty\left((0,t);H^{r+2}(\Omega)\right), u_t\in L^\infty\left((0,t);H^{r+1}(\Omega)\right)$, we have
\begin{align}
\label{ieq:I}
I&=||\eta_t||_\Omega^2\leq||\eta_t||_{\Th}^2\leq Ch^{2r+2}||u_t||_{H^{r+1}(\Omega)}^2,\\
\label{ieq:II}
II&=h^{-2}\beta||\eta||_\Omega^2\leq h^{-2}\beta||\eta||_{\Th}^2\leq Ch^{2r+2}||u||_{H^{r+2}(\Omega)}^2.
\end{align}
For the stabilization terms $III,IV$, we use the trace property \eqref{eq:pi:property1} and $J_s(u,v_h)=0$ with u sufficiently smooth. Then, we can obtain
\begin{align*}
\frac{1}{\rA}IV&= J_0(\p u,\p u)=J_0(\p u -u, \p u -u)= \sum_{F\in\Fr}\sum_{k=0}^{r}\omega_kh^{2k}[\partial^k (u-\p u)]^2_F\notag\\
&\leq 2\sum_{F\in\Fr}\sum_{k=0}^{r}\omega_k h^{2k}||\partial^k(u-\p u)||_F^2\notag\\
&\leq C'\sum_{F\in\Fr}\sum_{k=0}^{r}\omega_k h^{2k}h^{2(r+1)-2k-1}||u||_{H^{r+1}({\Omega})}^2\leq Ch^{2r+1}||u||_{H^{r+1}({\Omega})}^2.
\end{align*}
Similar to the analysis of the above inequality, we also have
\begin{align}\label{ieq:III}
\frac{1}{\rM}III=J_1((\p u)_t,(\p u)_t)\leq Ch^{2r+2}||u_t||_{H^{r+1}({\Omega})}^2.
\end{align}
Using the trace inequality \eqref{eq:pi:property1} and upwind flux $\hat{\eta}=\eta^{-}$, we have
\begin{align}\label{ieq:V}
V=\beta\sum_{j=2}^{N} \hat{\eta}_\jl^2+\beta\hat{\eta}^2_l\leq Ch^{2r+1}||u||_{H^{r+1}(\Omega)}^2.
\end{align}
Combing the above inequalities \eqref{ieq:0}-\eqref{ieq:V}, we can get
\begin{align}\label{ieq:error}
\frac{d}{dt}\left(||\theta||^2_\Omega+\rM J_1(\theta,\theta)  \right)\leq C_1(||\theta||_\Omega^2+\rM J_1(\theta,\theta))+C_2h^{2r+1}.
\end{align}
Thus, 
\begin{align}\label{ieq:error2}
\frac{d}{dt}e^{-C_1t}\left(||\theta||^2_\Omega+\rM J_1(\theta,\theta)  \right) \leq C_2e^{-C_1t}h^{2r+1},
\end{align}
which can be integrated in time to yield 
\begin{align}\label{ieq:error3}
||\theta||^2_\Omega+\rM J_1(\theta,\theta) \leq e^{C_1t}
\left(||\theta||^2_\Omega+\rM J_1(\theta,\theta) \right)|_{t=0}
+C_2\int_0^te^{C_1(t-\tau)}d\tau h^{2r+1}.
\end{align}
We note that the initialization of $u_h$ is based on the stabilized  $L^2$ projection $\Pi_h u_0$ \cite{burman2015stabilized} of the smooth initial data $u_0(x)$. Thus, combing the property of projection $\p$ with the error estimation of stabilized $L^2$ projection $\Pi_h$,  we have
\begin{align}\label{eq:error:initial}
||\theta||^2_\Omega|_{t=0}\leq||\p u_0-u_0||^2_\Omega+||u_0-\Pi_h u_0||^2_\Omega \leq Ch^{2r+2},\\
 J_1(\theta,\theta) |_{t=0}\leq 2 \left(J_1(\p u_0-u_0,\p u_0-u_0)+J_1(u_0-\Pi_h u_0, u_0-\Pi_h u_0)\right) \leq Ch^{2r+2}.\notag
\end{align}
Then, with initial error \eqref{eq:error:initial} and \eqref{ieq:error3}, we obtain
$$||\theta||^2_\Omega+\rM J_1(\theta,\theta) \leq Ch^{2r+1}.$$ 
Finally, applying  the triangle inequality and the property of projection $\p$, we have the  \textit{a priori} error estimate \eqref{ieq:error1}.
\end{proof}
\begin{rem}
Our estimate based on the projection $\pi_h^e$ yields lower  accuracy than for the standard method. In Section \ref{sec:Numerical} optimal accuracy is observed numerically.\end{rem}

\subsection{TVD stability}\label{sec:tvd}
In this subsection, we will prove that the stabilized  cut DG scheme \eqref{scheme:cutDG2} with piecewise constants in space for the linear advection equation \eqref{eq:adv:1d} with periodic boundary condition is TVD stable when the explicit Euler time discretization is applied. Let $u_j^n$ denote the solution $u_h(x,t_n)$ in the element $I_j\cap\Omega$. 
A DG scheme is TVD stable if it satisfies
\begin{equation}
TV\left({u_h}^{n+1}\right) \leq TV\left({u_h}^{n}\right), \quad \text { with } T V\left({u_h}^{n}\right)=\sum_{j}\left|{u}_{j+1}^{n}-{u}_{j}^{n}\right|.
\end{equation}

With the setting in Figure \ref{mesh_discretization}, we have the cut DG scheme \eqref{scheme:cutDG2} with piecewise constants in space and explicit Euler time discretization as 
\begin{align}
\label{eq:tvd:1}
\alpha  u_{1}^{n+1}+\rM  u_{1}^{n+1}-\rM  u_{2}^{n+1}=&\alpha  u_{1}^{n}-\rM  u_{2}^{n}+\rM  u_{1}^{n} \\
& +\lambda \left(u_N^{n}-u_{1}^{n}\right)+\lambda \gamma_{A}\left(u_{2}^{n}-u_{1}^{n}\right),\notag\\
\label{eq:tvd:2}
 u_{2}^{n+1}+\rM  u_{2}^{n+1}-\rM  u_{1}^{n+1}=&  u_{2}^{n}-\rM  u_{1}^{n}+\rM  u_{2}^{n} \\
& +\lambda \left(u_{1}^{n}-u_{2}^{n}\right)-\lambda \gamma_{A}\left(u_{2}^{n}-u_{1}^{n}\right),\notag\\
 u_{j}^{n+1}=& u_{j}^{n}+\lambda\left(u_{j-1}^{n}-u_{j}^{n}\right),\quad j=3,\cdots,N.
\label{eq:tvd:3}
\end{align}
Here, $\lambda={\Delta t}/{h}$ is Courant number. For this scheme, we have the following theorem.
\begin{thm}
The stabilized cut DG scheme \eqref{scheme:cutDG2} 
with piecewise constants in space and the explicit Euler time discretization for linear advection equation with $\beta=1$ is TVD stable under the condition that the time step satisfies $\frac{\Delta t}{h}\leq\alpha+\frac{\rM}{\rM+1}$, and parameters satisfy $0<\rM\leq\rA $ and $(1+\alpha)\rA -\rM\leq 1-\alpha$.
\end{thm}
\begin{proof}
With the periodic boundary condition we define $u_0=u_N$ and  divide
 \begin{align*}
 \sum_{j=1}^N\left|u_{j}-u_{j-1}\right|=& {\sum_{j= 3}^N\left|u_{j+1}-u_{j}\right|}+{\left|u_{2}-u_{1}\right|}
 +{\left|u_{3}-u_{2}\right|}
 +{\left|u_{1}-u_N\right|}.
 \end{align*}
We first consider $u_{2}^{n+1}-u_{1}^{n+1}$ by subtracting equation \eqref{eq:tvd:1} from equation \eqref{eq:tvd:2} 
 multiplied by $\alpha$, 
 \begin{equation}
u_{2}^{n+1}-u_{1}^{n+1}=\left( 1- \lambda \frac{\alpha+\alpha  \rA+  \rA}{\alpha+\alpha \rM+\rM}\right)\left(u_{2}^{n}-u_{1}^{n}\right)+\frac{\lambda}{\alpha+\alpha \rM+\rM}\left(u_{1}^{n}-u_{N}^{n}\right).
\label{eq:tvd:5}
\end{equation}
 Next, we compute $u_{3}^{n+1}-u_{2}^{n+1}$ by subtracting equation \eqref{eq:tvd:2} from equation \eqref{eq:tvd:3} with $j=3$ and get 
 \begin{equation}
u_{3}^{n+1}-u_{2}^{n+1}=\rM\left(u_{2}^{n+1}-u_{1}^{n+1}\right)+(1-\lambda)\left(u_{3}^{n}-u_{2}^{n}\right)+\left(\lambda+\lambda \rA-\rM\right)\left(u_{2}^{n}-u_{1}^{n}\right).
\label{eq:tvd:6}
\end{equation}
Replacing $u_{2}^{n+1}-u_{1}^{n+1}$ in equation \eqref{eq:tvd:6} by equation \eqref{eq:tvd:5}, we have 
\begin{align}
u_{3}^{n+1}-u_{2}^{n+1}=&\lambda \frac{\alpha+ \rM+\alpha \rA}{\alpha +\rM+\alpha\rM}\left(u_{2}^{n}-u_{1}^{n}\right)+(1-\lambda)\left(u_{3}^{n}-u_{2}^{n}\right)\notag\\
&+\frac{\lambda\rM}{\alpha+\alpha\rM+\rM}\left(u_{1}^{n}-u_{N}^{n}\right).
\label{eq:tvd:7}
\end{align}
Then, we consider $u_1^{n+1}-u_N^{n+1}$. By subtracting $\alpha$ times equation \eqref{eq:tvd:3} with $j=N$ from equation \eqref{eq:tvd:1}, and replacing 
$u_{2}^{n+1}-u_{1}^{n+1}$  by equation \eqref{eq:tvd:5}, we obtain
\begin{align}\label{eq:tvd:8}
u_1^{n+1}-u_N^{n+1}=& \frac{ \lambda\left(  \rA-\rM\right)}{\alpha+\alpha \rM+\rM}\left(u_{2}^{n}-u_{1}^{n}\right)+\left(1-\lambda \frac{\rM+1}{\alpha+\alpha \rM+\rM}\right)\left(u_{1}^{n}-u_N^{n}\right)\notag\\
&-\lambda (u_{N-1}^n-u_N^n).
\end{align}
 For the standard part with $4\leq j\leq N$, we have 
 \begin{align}\label{eq:tvd:9}
 u_{j}^{n+1}-u_{j-1}^{n+1}=(1-\lambda)(u_j^n-u_{j-1}^n)+\lambda(u_{j-1}^n-u_{j-2}^{n}).
 \end{align}
 Summing the absolute of equations \eqref{eq:tvd:5}, \eqref{eq:tvd:7}, \eqref{eq:tvd:8} and \eqref{eq:tvd:9}, we get
 \begin{align*}
 \sum_{j=1}^N&\left|u_{j}^{n+1}-u_{j-1}^{n+1}\right|\leq
\left(\left|1- \frac{\lambda(\rM+1)}{\alpha+\alpha \rM+\rM}\right|+\frac{\lambda(\rM+1)}{\alpha +\rM+\alpha\rM}\right)\left|u_{1}^{n}-u_N^{n}\right|\\
 &+\left(\left|1- \lambda \frac{\alpha+\alpha  \rA+  \rA}{\alpha+\alpha \rM+\rM}\right|+\left|\frac{ \lambda( \rA-\rM)}{\alpha+\alpha \rM+\rM}\right|+\lambda \frac{\alpha+ \rM+\alpha \rA}{\alpha +\rM+\alpha\rM}\right)\left|u_{2}^{n}-u_{1}^{n}\right|\\
 &+\sum_{j=3}^{N}|u_j^n-u_{j-1}^n|.
 \end{align*}
 In the above inequality, the parameters $\rM\ge 0,\rA\ge 0$ and $0<\lambda\le 1$ are used.  With $\rM\leq \rA$ and $(1+\alpha) \rA-\rM\leq 1-\alpha$, we have $\lambda\frac{\alpha+\alpha \rA+ \rA}{\alpha+\alpha\rM+\rM}\leq \lambda\frac{\rM+1}{\alpha+\alpha\rM+\rM}\leq 1$ under the time step satisfying $\lambda=\frac{\Delta t}{h}\leq\alpha+\frac{\rM}{\rM+1}$. Thus,  we have
\begin{align}
1- \frac{\lambda(\rM+1)}{\alpha+\alpha \rM+\rM}\ge 0, 1- \lambda \frac{\alpha+\alpha  \rA+  \rA}{\alpha+\alpha \rM+\rM}\ge 0,  \rA-\rM\ge 0,
\end{align}
 and we can get
\begin{align*}
\sum_{j=1}^N\left|u_{j}^{n+1}-u_{j-1}^{n+1}\right|\leq\sum_{j=3}^{N}|u_j^n-u_{j-1}^n|+|u_1^n-u_N^n|+|u_2^n-u_1^n|=\sum_{j=1}^N|u_j^n-u_{j-1}^n|.
\end{align*}
\end{proof}

As for the standard  DG scheme, we  can not get the TVD stability for higher order polynomials. Similar to the standard DG method, we use the TVB \textit{minmod} limiter in \cite{ShuDG2} to control the oscillations and overshoot produced from the stabilized cut DG scheme \eqref{scheme:cutDG2}. We define  the cell average of the solution $u$ as
$$
\overline{u}_{j}=\frac{1}{|I_j\cap \Omega|} \int_{I_{j}\cap \Omega} u dx,
$$
and further define
$$
\tilde{u}_{j}=u_{j+\frac{1}{2}}^{-}-\overline{u}_{j}, \quad \widetilde{\widetilde{u}}_{j}=\overline{u}_{j}-u_{j-\frac{1}{2}}^{+},
\Delta_{+} \overline{u}_{j}=\overline{u}_{j+1}-\overline{u}_{j}, \quad \Delta_{-} \overline{u}_{j}=\overline{u}_{j}-\overline{u}_{j-1}.
$$
We modify  both $\tilde{u}_{j}$ and $\tilde{\tilde{u}}_{j}$  by the TVB \textit{minmod} limiter $\tilde{m}$, 
\begin{equation}
\tilde{u}_{j}^{(\mathrm{mod})}=\tilde{m}\left(\tilde{u}_{j}, \Delta_{+} \overline{u}_j, \Delta_{-} \overline{u}_j\right), \quad \widetilde{\widetilde{u}}_{j}^{(\mathrm{mod})}=\tilde{m}\left(\widetilde{\widetilde{u}}_{j}, \Delta_{+} \overline{u}_j, \Delta_{-} \overline{u}_j\right).
\end{equation}
For the definitions of function $\tilde{m}$  and details of the TVB limiter, we refer to \cite{ShuDG2}. 
Then we recover the limited function $u_{h}^{(\mathrm{mod})}$ by maintaining the old cell average $\bar{u}_j$ and the new point values  given by $
u_{h}^{(mod)}\left(x_{j+\frac{1}{2}}^{-}\right)=\bar{u}_{j}+\tilde{u}_{j}^{(m o d)}, \quad u_{h}^{(mod)}\left(x_{j-\frac{1}{2}}^{+}\right)=\bar{u}_{j}-\widetilde{\widetilde{u}}_{j}^{(mod)}.$ 
This recovery is unique for $P^{k}$ polynomials with $k \leq 2$. When $k>2$, the recovery is done by setting high order coefficients than $2$ to zero.

Note that we can not prove the TVDM property of the stabilized cut DG scheme with high order polynomials. The reason is  the jump of high order derivatives in the stabilization $J_s(u_h,v)$. However, we observe numerically that the stabilized cut DG scheme is TVB  when the limiter is applied. 
On coarse meshes with one small cut element and meshes with more small cut elements, some oscillations are triggered when a discontinuity passes a cut element and its neighbours. In appendix B, we describe a modified limiting procedure, where the approximation is locally reduced to the $P^0$ scheme in a cut element and its neighbours when the standard limiter indicates that limiting is needed in a cut element or its neighbour. 

\section{Numerical examples}
\label{sec:Numerical}

In this section, we present numerical examples that demonstrate the performance  of our proposed stabilized cut DG scheme \eqref{scheme:cutDG2} for scalar problems \eqref{eq:linear:1d}. Based on the numerical studies for eigenvalues of the spatial operator in Section 3 and the analysis in Section 4
we include both stabilization terms  $J_1(u_t,v)$ and $J_0(u,v)$ with coefficients $\gamma_M=0.25$, $\gamma_A=0.75, \omega_k=\frac{1}{(2k+1)k!^2}$.
Both linear and nonlinear cases are considered. 
In all our computations, the third order  TVD RK method \cite{Time_shu} is used when $r\leq 2$, and the fourth order five stages RK method is used when $r=3$  for the time discretization. 
We use a time step $\Delta t=\lambda h$, where the Courant number $\lambda$ varies with polynomial order. We consider the case with a cut element at the boundary, see Figure \ref{mesh_discretization}, and cases with one or more cut elements in the interior, see Figure \ref{fig:cutmiddle1}. In the problems with non-smooth solutions, limiting is used to control oscillations.

\subsection{Accuracy test for the linear case}
In this subsection, we consider the one dimensional linear advection equation and periodic boundary condition,  
\begin{align}\label{eq:test:linear}
&u_t+u_x=0, \, 0<x<2, \, t>0,
\end{align}
with the initial data $u(x,0)=1.0+0.5\sin(\pi x)$. The exact solution is $u(x,t)=1+0.5\sin(\pi(x-t))$. We test the problem \eqref{eq:test:linear} on the mesh setting in Figure \ref{mesh_discretization} with one cut element on the left boundary.   The computational domain $[x_L,x_R]=[x_{l}-(1-\alpha) h, x_{r}]$ with $\alpha\in (0,1]$  and $h$ the regular size of interior elements.  In our test, $\Delta t=\lambda\frac{x_{r}-x_{l}}{N+\alpha-1}$ with $N=40,80,160,320,640$ and $\lambda=0.5,0.3,0.2,0.14$ for $r=0,1,2,3$ respectively.
The stabilized cut DG scheme \eqref{scheme:cutDG2} with upwind flux is applied up to time $t=1$. In Figure \ref{fig:table:linears2}, we plot $L^2$ and $L^\infty$ errors 
of the numerical solutions on the mesh with a cut element of size $\alpha=10^{-4}$ and the corresponding average convergence rates are shown in the legend.  Note that the stabilized cut DG schemes are stable and converge optimally with the same Courant number $\lambda$ as the standard DG scheme.  We have also tested this problem on the mesh in 
Figure \ref{fig:cutmiddle1}, and with $\alpha$'s as small as $10^{-10}$. In all cases the results are very similar.

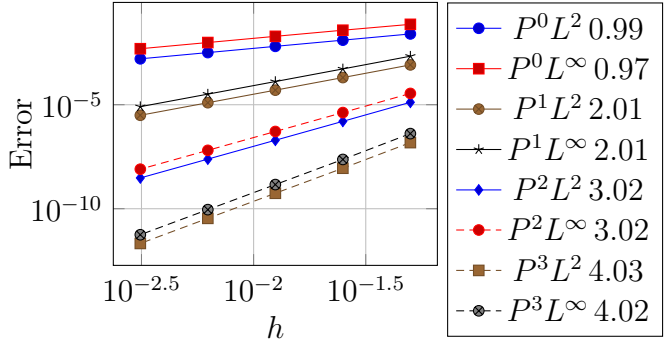
\begin{figure}[tbhp]
  \centering
   \begin{tikzpicture}
    \begin{loglogaxis}
    [height=2.0in, grid=major,
      xlabel={$h$}, ylabel={Error},
      legend entries={$P^0 L^2 \,0.99$,$P^0 L^\infty \,0.97$,$P^1 L^2 \,2.01$,$P^1 L^\infty\,2.01$,$P^2 L^2 \,3.02$,$P^2 L^\infty \,3.02$,$P^3 L^2\,4.03$,$P^3 L^\infty\,4.02$
    }, 
      legend pos=outer north east
       ]
\addplot table [x=h, y=P0_L2] {Linear_Minus4_error_boundarycut.dat};
\addplot table [x=h, y=P0_Linf] {Linear_Minus4_error_boundarycut.dat};
\addplot table [x=h, y=P1_L2] {Linear_Minus4_error_boundarycut.dat};
\addplot table [x=h, y=P1_Linf] {Linear_Minus4_error_boundarycut.dat};
\addplot table [x=h, y=P2_L2] {Linear_Minus4_error_boundarycut.dat};
\addplot table [x=h, y=P2_Linf] {Linear_Minus4_error_boundarycut.dat};
\addplot table [x=h, y=P3_L2] {Linear_Minus4_error_boundarycut.dat};
\addplot table [x=h, y=P3_Linf] {Linear_Minus4_error_boundarycut.dat};
 \end{loglogaxis}
 \end{tikzpicture}
  \caption{$L^2$  and $L^{\infty}$  errors  for the stabilized scheme \eqref{scheme:cutDG2} of different polynomial orders 
  for the advection problem \eqref{eq:test:linear} at $t=1$. The average convergence rate is given in the legend.
} \label{fig:table:linears2}
\end{figure}

\subsection{The linear case with non-smooth data}
In this subsection, we consider the linear advection problem $u_t+u_x=0$ with non-smooth initial data and a non-smooth boundary condition.  
We have tested many more $\alpha$ values than those presented, and the solutions behave similarly for all $\alpha$ values.

\subsubsection{Non-smooth initial data}
Here, we test the advection problem \eqref{eq:test:linear} with periodic boundary condition on  domain $\Omega=[0,1]$ and non-smooth initial data
\begin{equation}\label{eq:nonsmoothinitial}
u(x,0)=
\left\{\begin{array}{ll}
{1} & {0.1<x<0.5}, \\
{0} & {\text{otherwise.}}
\end{array}\right.
\end{equation}
 We use a mesh partition with equal size $h=(x_r-x_l)/N$, with the middle element $[0.5,0.5+h]$ split into two cut elements $[0.5,0.5+\alpha h],[0.5+\alpha h, 0.5+h]$ with $\alpha=10^{-4}$,  as in Figure \ref{fig:cutmiddle1}. 
We first solve this problem using the cut DG scheme \eqref{scheme:cutDG2} with $P^0$ approximation and upwind flux,  and the Courant number $\lambda=0.2$.  The stabilization is added only on the  interior interface point $x=0.5$ since $\alpha<0.5$. In Figure \ref{fig:nonsmoothinitial:p1:overshoot} (a), we show the solution based on $P^0$ cut DG scheme.  We do not observe any overshoots from our stabilized cut DG scheme. This is expected by the TVD stability. 

To show the performance of the high order cut DG scheme, we also test this problem with $P^1$ polynomial space and $\alpha=10^{-4}$, $\lambda=0.3$.  
In Figure \ref{fig:nonsmoothinitial:p1:overshoot} (b)-(d), we show the numerical solution $u_h$ and the total variation of the mean values  $TV(\bar{u}_h)$  of $u_h$  for different element sizes.  We can see that on the coarsest mesh the numerical solution has an overshoot near the cut element and that $TV(\bar{u}_h)$ decreases after the discontinuity passes through the cut element.  On the finer meshes we have not seen any overshoots.  In the Figure \ref{fig:nonsmoothinitial:p1:overshoot} (d), we also study the performance of the stabilized DG scheme in a long time simulation. Note that there are many small  increases in total variation $TV(\bar{u}_h)$ (red line) when the standard limiting is applied, each such increase decays rapidly with time and $TV(\bar{u}_h)$ remains  bounded. We believe the overshoots may be caused by the higher derivative terms in the stabilization term $J_s(u,v)$. 
With the modified  limiting descried in appendix B,  the results improve, see Figure \ref{fig:nonsmoothinitial:p1:overshoot} (c). There are no overshoots even on the coarsest mesh, and $TV(\bar{u}_h)$ (green line in (d)) is diminishing monotonically. Comparing with the standard DG method (blue line in (d)),  we note that our proposed cut DG methods are more dissipative. We note that \textit{limiter} and \textit{modified}  in the legend of figures mean the standard limiting and modified limiting is applied to the scheme, respectively.
\begin{figure}[tbhp]
  \centering
\subfigure[$P^0$ ]{
\includegraphics[width=2.2in]{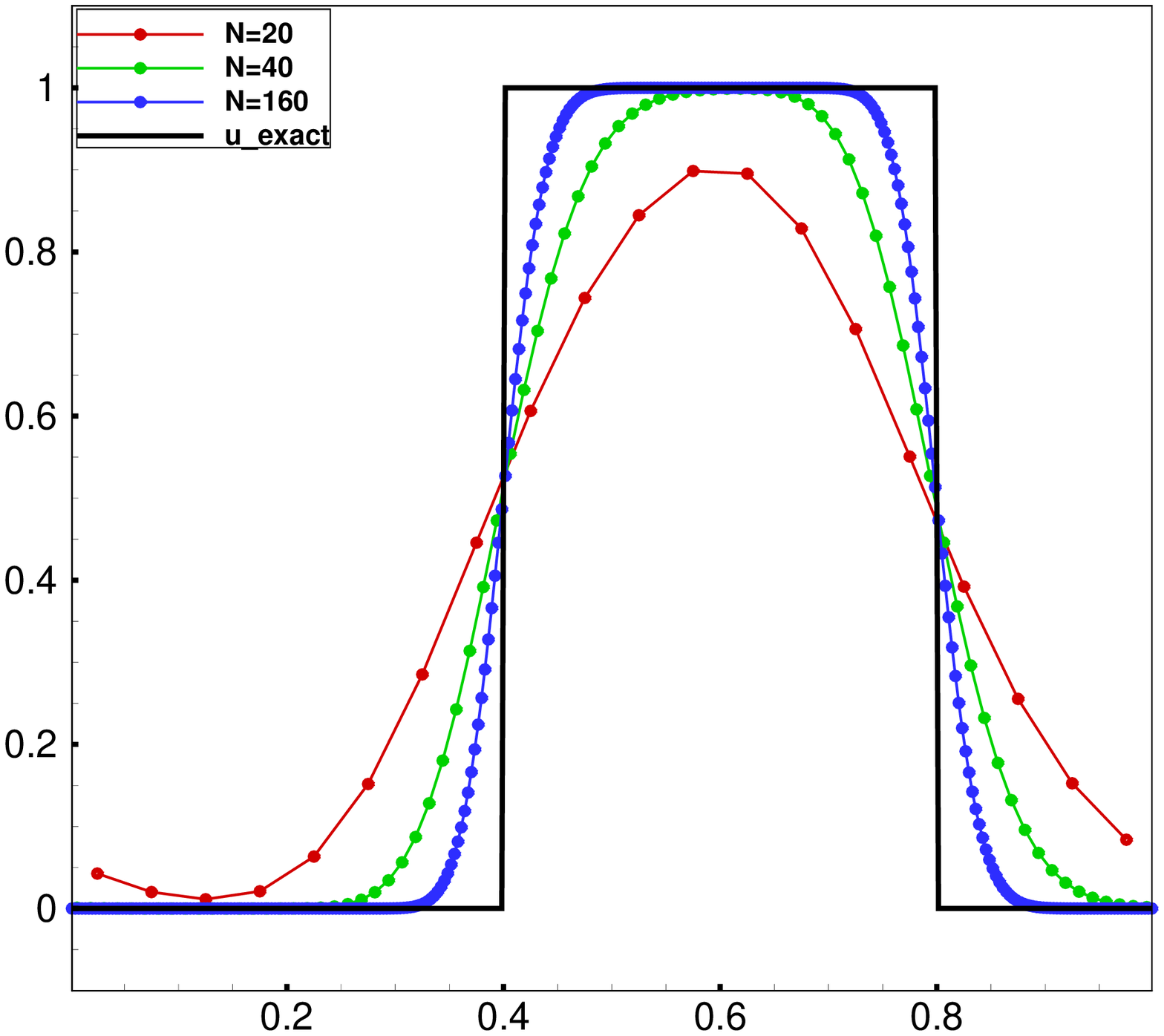}
}
\subfigure[$P^1$ with TVB \textit{limiter}]{
\includegraphics[width=2.2in]{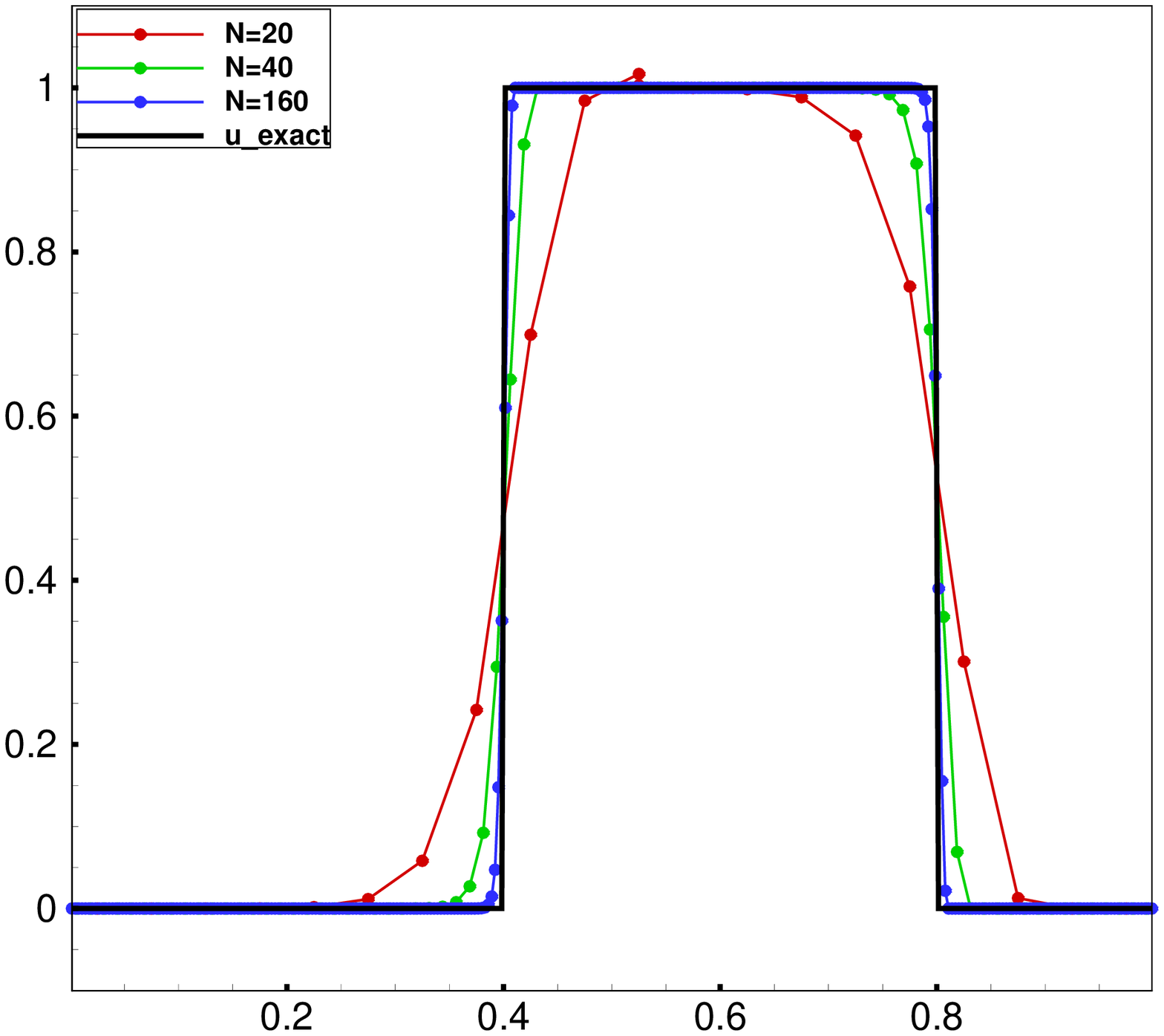}
}
\subfigure[$P^1$ with \textit{modified} limiting]{
\includegraphics[width=2.2in]{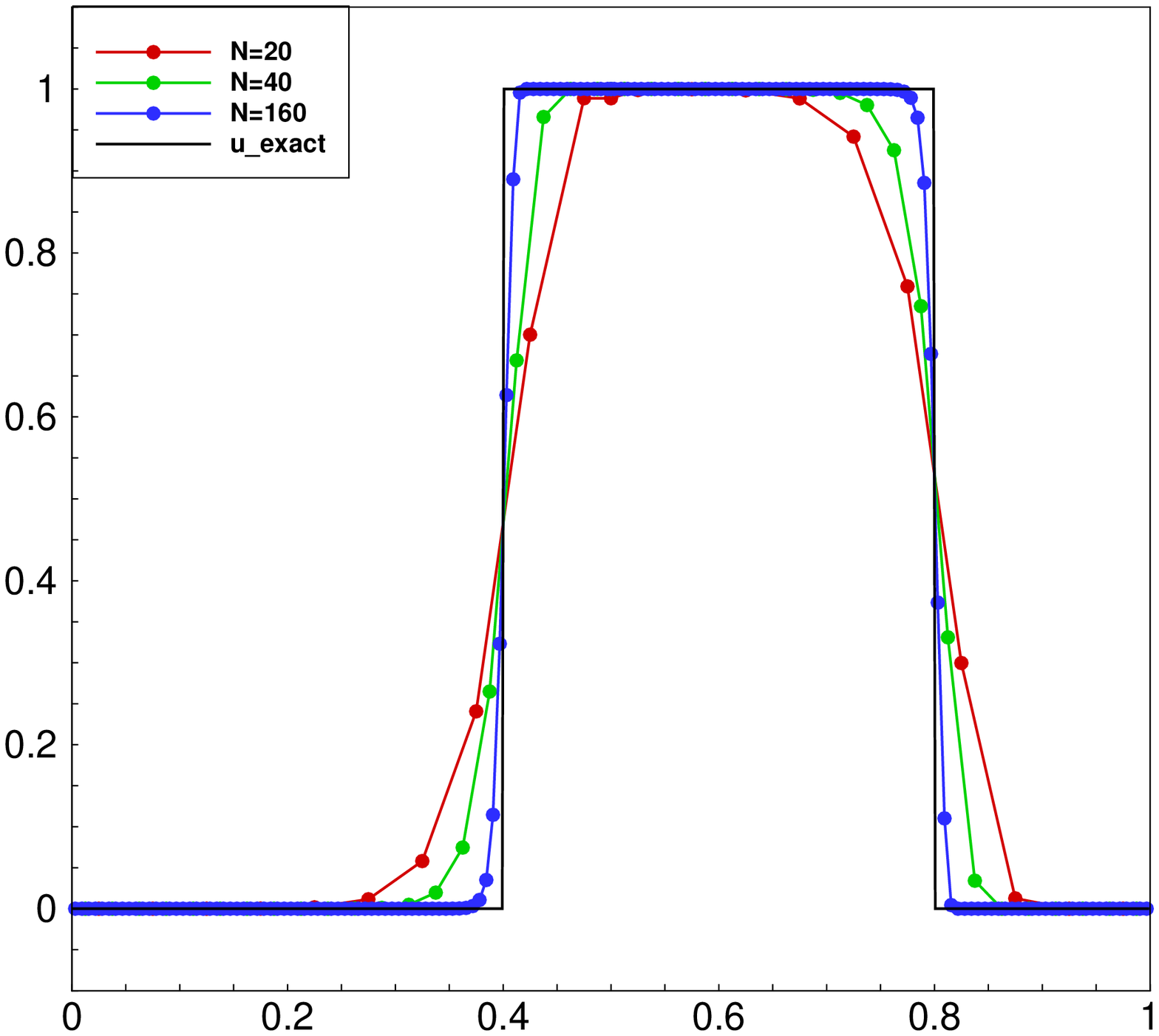}
}
\subfigure[$TV(\overline{u})$ up to $t=100$ with $N=20$]{
\includegraphics[width=2.2in]{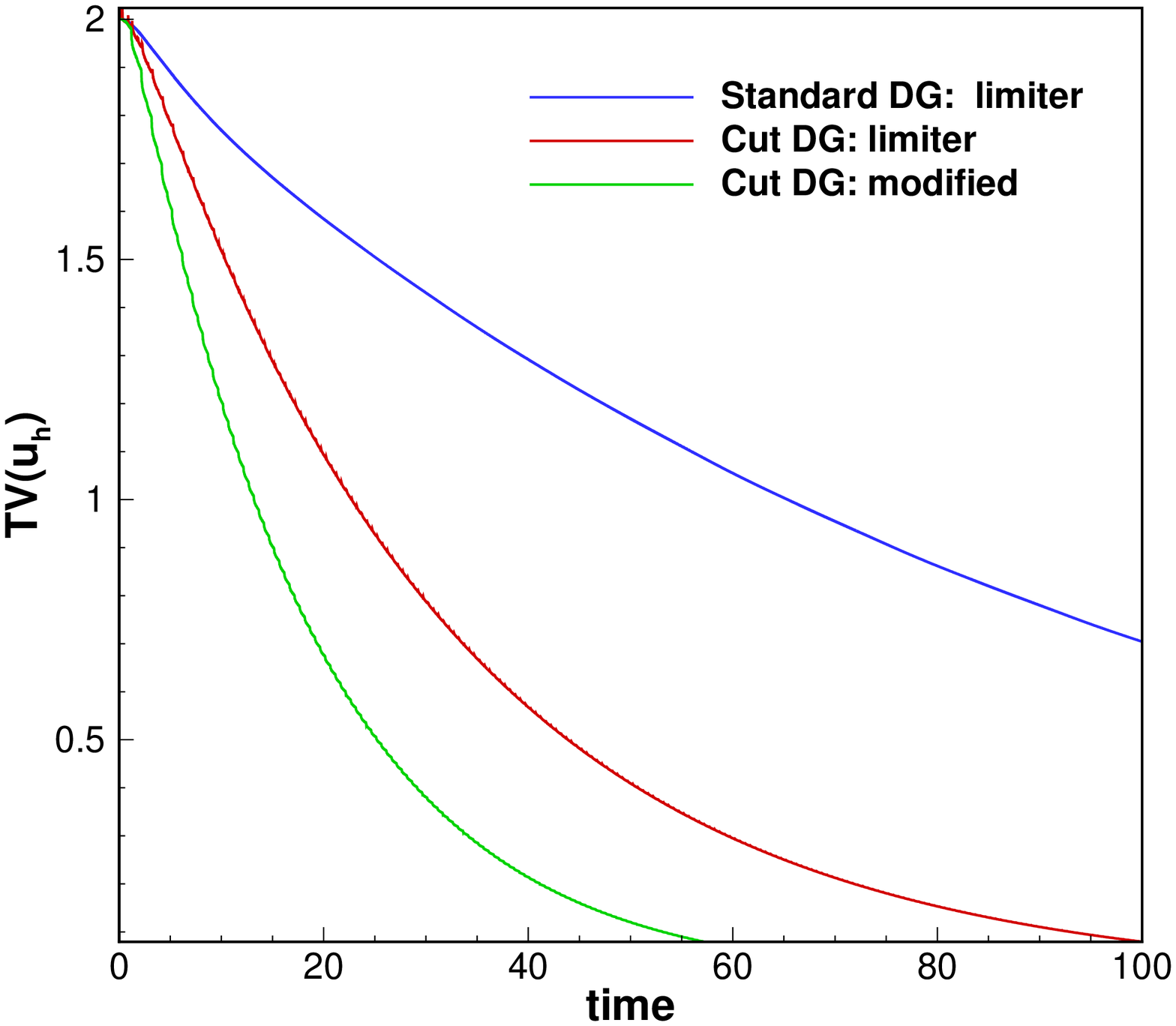}
}
\caption{ Numerical solutions for the advection equation  with non-smooth initial data at $t=0.3$ (a-c),  and $TV(\overline{u_h})$ as function of t (d). Results are for one small cut element with $\alpha=10^{-4}$.  
 }\label{fig:nonsmoothinitial:p1:overshoot}
\end{figure}

\subsubsection{Non-smooth boundary data}
Next, similar to computations in \cite{tan2010inverse}, we test the advection equation \eqref{eq:test:linear} on the physical domain $[x_l,x_r]=[0,2]$ with the left boundary condition being
\begin{equation}\label{exe:eq:nonsmoothBC}
g(t)=\left\{\begin{array}{ll}{0,} & {t \leqslant 1,} \\ {-1,} & {t>1.}\end{array}\right.
\end{equation}
We solve this example by the stabilized cut DG  scheme \eqref{scheme:cutDG2} with  $P^2$ polynomials and upwind flux. Here the time step is $\Delta t=0.2h$ with $h=(x_r-x_l)/(N+1-\alpha)$.  We use the mesh setting in Figure \ref{mesh_discretization}, with a cut element with $\alpha=10^{-2}$ located on the left boundary. From the results in Figure \ref{exe:nonsmoothBC}, we can observe that the cut DG scheme with sufficient mesh refinement can simulate this problem well and capture the discontinuity. However, also for this case, we observe undershoots on coarse meshes. These decay with time after the discontinuity passes the cut element. The phenomena is not seen on fine meshes.

\begin{figure}[tbhp]
  \centering
 \includegraphics[width=2.2in]{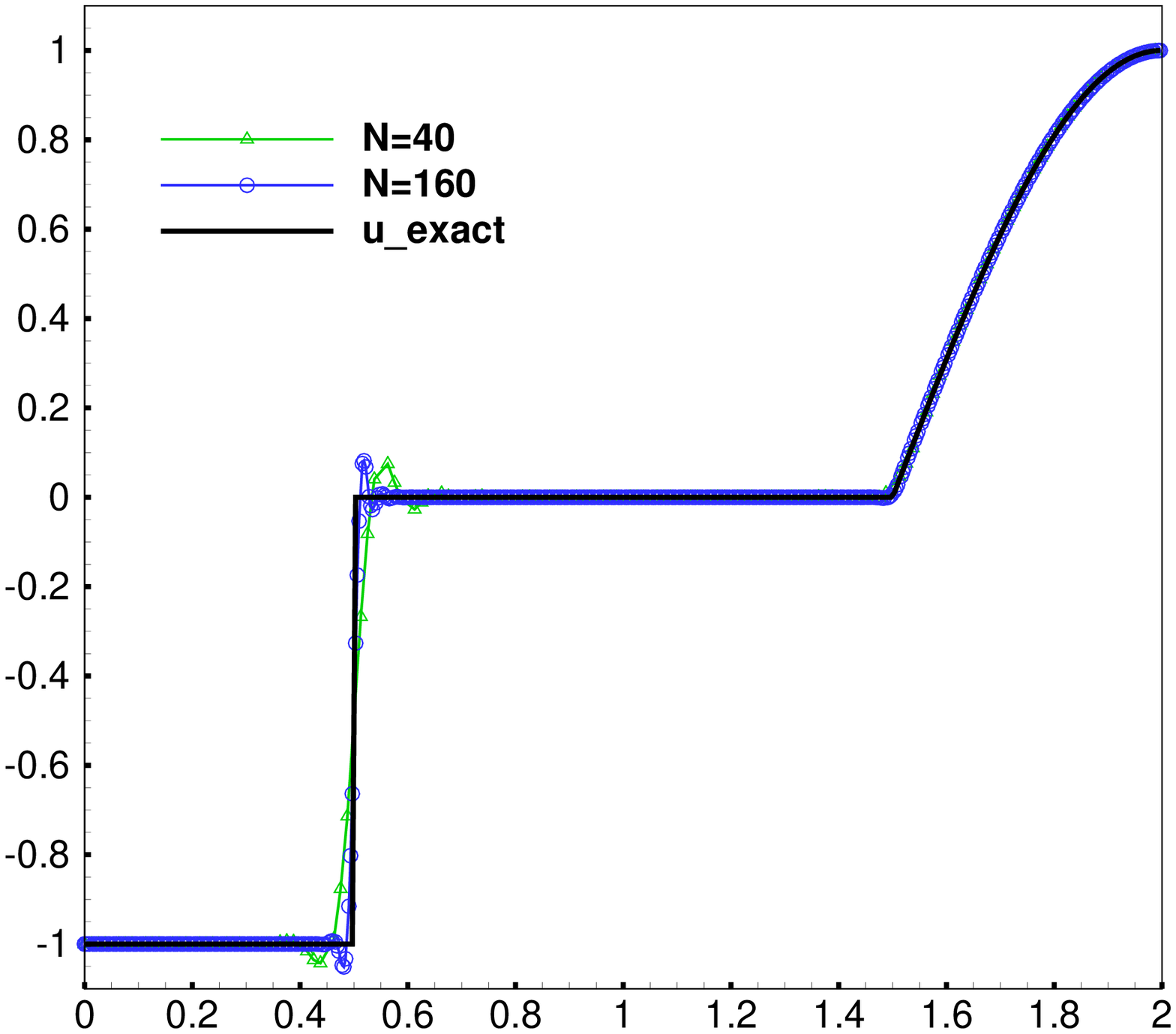}
  \includegraphics[width=2.2in]{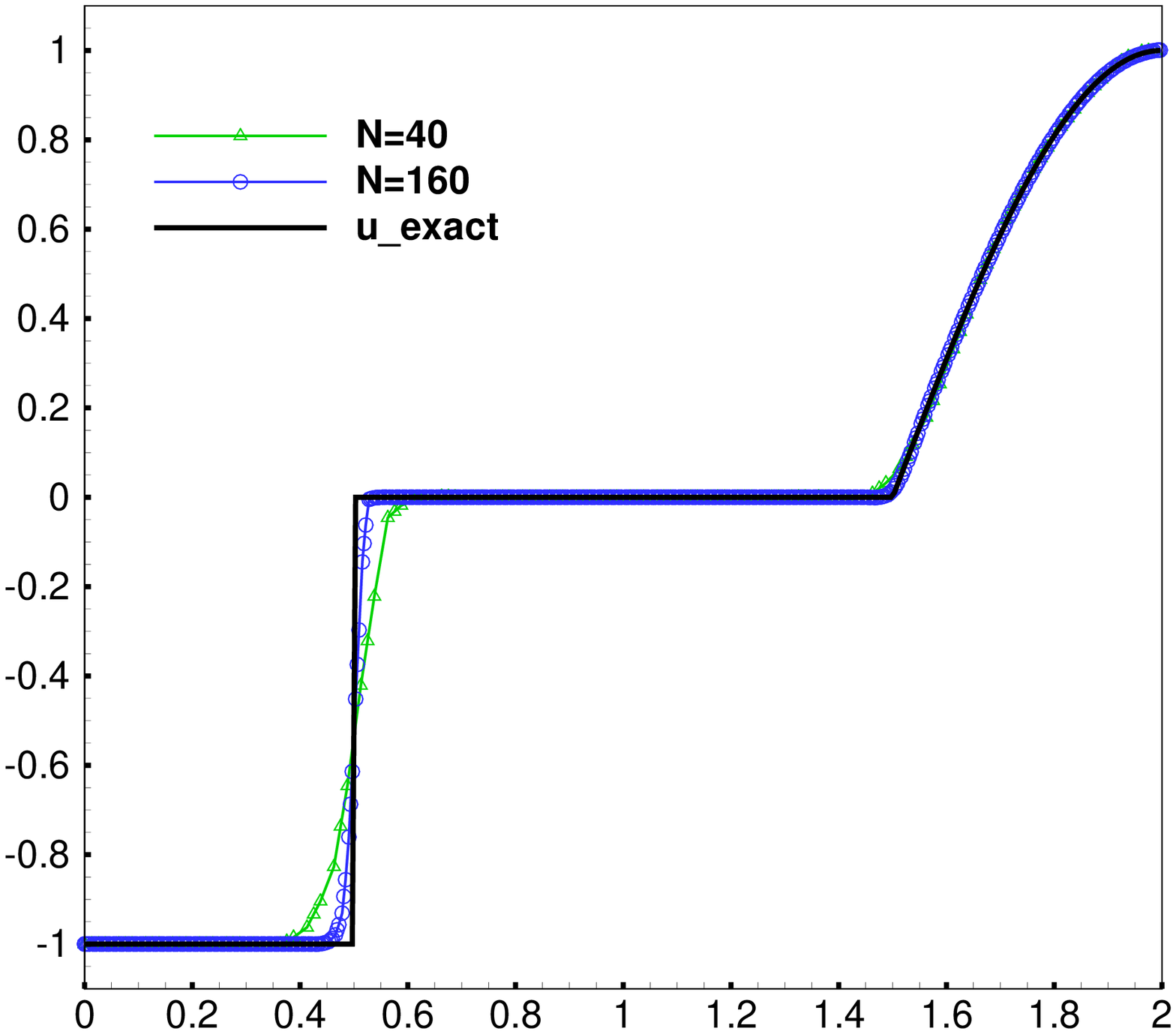}
  \caption{Numerical $P^2$ solutions for the advection equation with non-smooth boundary data, and one small cut element with $\alpha=10^{-2}$ at $x=0$. Left: without limiter, right: with limiter. 
 }\label{exe:nonsmoothBC}
\end{figure}

\subsection{Nonlinear case: Burgers' equation}
In this subsection, we apply the high order stabilized cut DG scheme \eqref{scheme:cutDG2} with Godunov flux to the Burgers' equation 
\begin{align}\label{eq:Burges}
&u_t+\left(\frac{u^2}{2}\right)_x=0,\quad x\in [x_l,x_r], \; t>0.
\end{align}
In these computation we used Courant number $\lambda=\alpha+\frac{\rM}{\rM+1},0.3,0.2,0.1$ for polynomial spaces $P^r$ with $r=0,1,2,3$,  respectively. In all computation for Burgers' equation, we use the mesh setting  in Figure \ref{fig:cutmiddle1} with many cut elements  located in an interval in 
the central part of the domain. The cut elements are created by splitting
each regular element in the interval into one small cut element of size $\alpha_k h$ and another of size $(1-\alpha_k)h$. Here $\alpha_k=s\alpha$ with $\alpha=10^{-4}$ and $s$ a random number  in $ [0.01,1]$.

\subsubsection{Smooth initial data}
We first test the Burgers' equation \eqref{eq:Burges} with smooth  initial data $u_0(x)=\sin(\pi x), x\in[0,2]$ and periodic boundary condition. This problem has a known solution, which we use as a reference when computing errors. The solution is initially smooth, but at $t=\frac{1}{\pi}$ a shock forms at $x=1$.  We compute the problem for accuracy until time $t=0.2$, which is before the shock appears.  The uniform meshes with $N=40,80,160,320,640$ elements are used as the background mesh. The cut elements are located in $[0.75,1.25]$ with $N/4$ small cut elements. In Figure \ref{table:burges2} the errors are shown and the slope of the error lines are given in the legend for $r=0,1,2,3$.   Observe that our method has optimal accuracy also in this case.  We also ran this problem with more small cut elements, located on $[0.5,1.5]$, with very similar result.
\begin{figure}[tbhp]
  \centering
  \begin{tikzpicture}
    \begin{loglogaxis}
    [height=1.9in, grid=major,
      xlabel={$h$}, ylabel={Error},
      legend entries={$P^0 L^2\,0.90$,$P^0 L^\infty \,0.89$,$P^1 L^2 \,2.01$,$P^1 L^\infty \,1.90$,$P^2 L^2 \,3.04$,$P^2 L^\infty \,2.71$,$P^3 L^2 \,4.06$, $P^3 L^\infty \,3.74$}, 
      legend pos=outer north east
       ]
\addplot table [x=h, y=P0_L2] {Burgers_cutmorecells_Minus4_error_Nover4.dat};
\addplot table [x=h, y=P0_Linf] {Burgers_cutmorecells_Minus4_error_Nover4.dat};
\addplot table [x=h, y=P1_L2] {Burgers_cutmorecells_Minus4_error_Nover4.dat};\addplot table [x=h, y=P1_Linf]{Burgers_cutmorecells_Minus4_error_Nover4.dat};\addplot table [x=h, y=P2_L2]{Burgers_cutmorecells_Minus4_error_Nover4.dat};
\addplot table [x=h, y=P2_Linf] {Burgers_cutmorecells_Minus4_error_Nover4.dat};
\addplot table [x=h, y=P3_L2] {Burgers_cutmorecells_Minus4_error_Nover4.dat};\addplot table [x=h, y=P3_Linf] {Burgers_cutmorecells_Minus4_error_Nover4.dat};
 \end{loglogaxis}
 \end{tikzpicture}
  \caption{$L^2$ and $L^{\infty}$  errors of  $u_h$ from scheme \eqref{scheme:cutDG2} on the mesh with $\alpha=10^{-4}$ and cut elements in $[0.75,1.25]$ for the Burgers' equation \eqref{eq:Burges} with periodic boundary condition at $t=0.2$. }
  \label{table:burges2}
\end{figure}
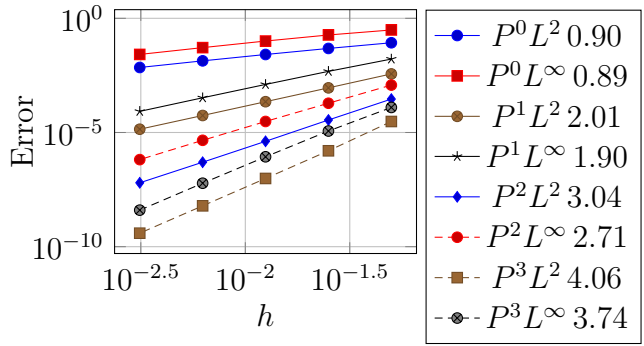

Next we solved the Burgers' equation with $P^2$ approximation with  cut elements in $[0.75,1.25]$ until  time $t=0.5$,  when the shock has been formed. Results for $h=1/40,1/160$ (corresponding to $20, 80$ small cut elements) with standard and modified limiting   are shown in the Figure \ref{figure:burgesshock2}.  With standard limiting there are  some overshoots and oscillations near the shock and at the cut elements. With the modified limiting in Appendix B, there are no such artefacts. 

\begin{figure}[tbhp]
  \centering
   \includegraphics[width=2.2in]{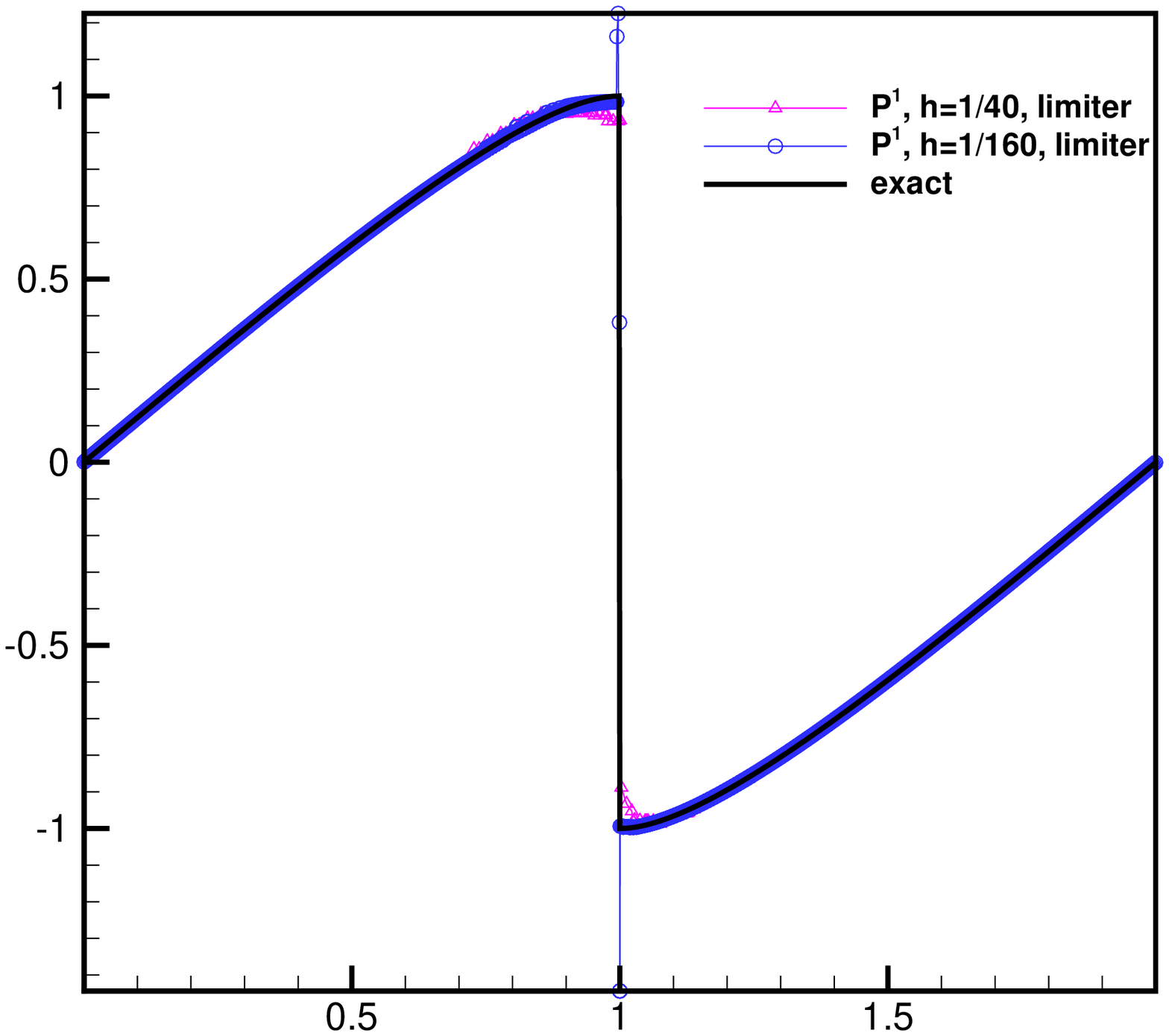}
 \includegraphics[width=2.2in]{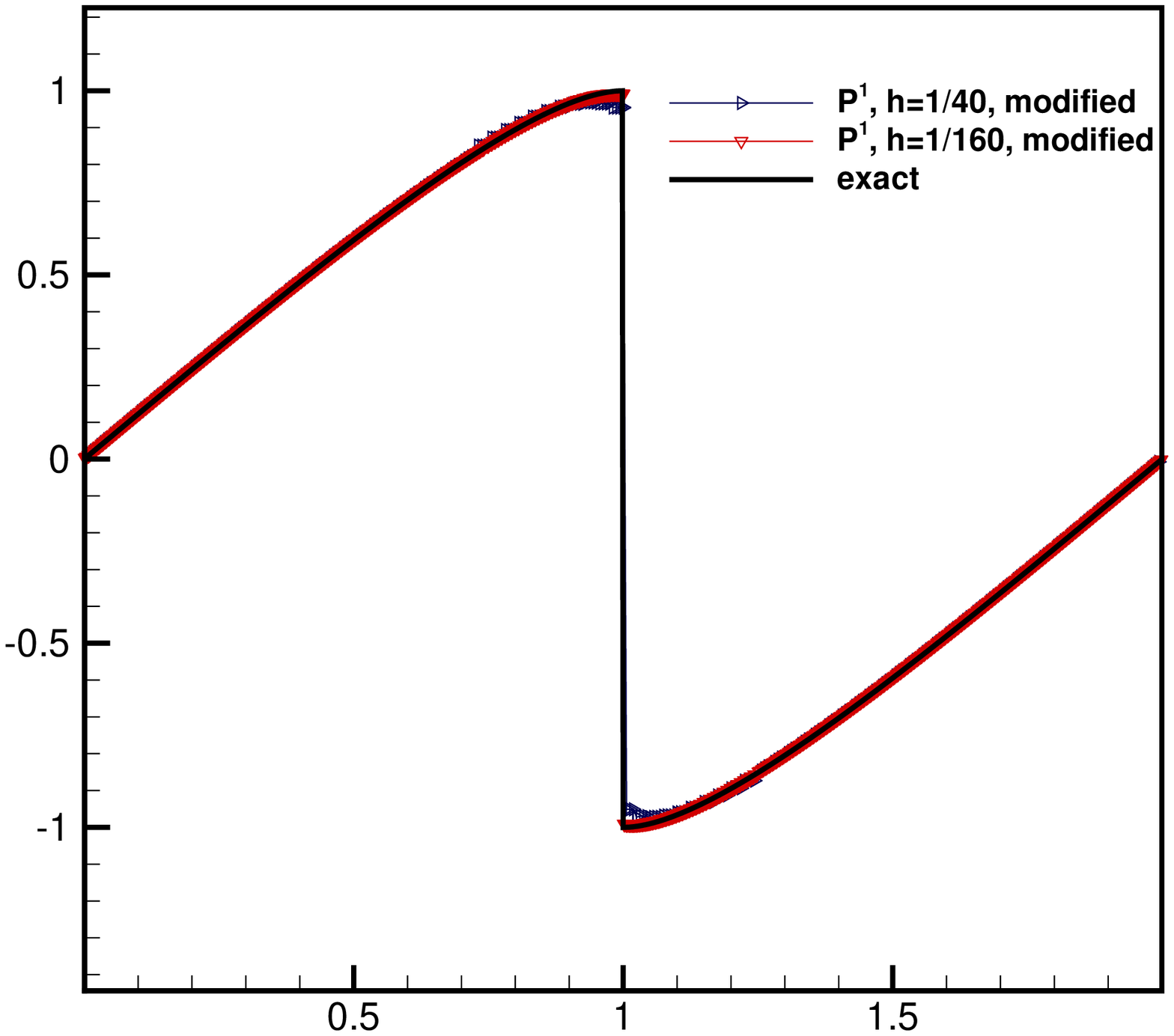}
  \caption{The numerical solutions of Burgers' equation \eqref{eq:Burges} at $t=0.5$,  by the $P^2$ scheme \eqref{scheme:cutDG2}, with standard (left) and modified (right) limiting. All elements in $[0.75,1.25]$ are cut.}\label{figure:burgesshock2}
\end{figure}

\subsubsection{Riemann problems}
Consider Burgers' equation \eqref{eq:Burges}  with initial data
\begin{align*}
u_0(x)=
\left\{\begin{array}{ll}
{u_l,} & {x\leq0,} \\
{u_r,} & {x>0.} 
\end{array}\right.
\end{align*}
We will use our proposed stabilized cut DG scheme \eqref{scheme:cutDG2} with the $P^0$ and $P^1$ approximations, 
and let all elements in $[-0.5,0.5]$ be cut.

First we let $u_l=-1<0<u_r=1$. In this case, a rarefaction wave is the weak solution, which satisfies the entropy condition. We solve this problem up to time $t=0.5$ with an outflow boundary condition on the left side and an inflow boundary on the right side. The results are shown in Figure \ref{figure:Rarefaction}. The cut DG scheme based on $P^0$ can simulate this rarefaction wave well. For the $P^1$ approximation, overshoots are observed near the contact points, $-0.5$ and $0.5$, when no limiter is applied. With the standard TVB limiter, these oscillations disappear.
However, oscillations are instead introduced in cut elements inside the rarefaction. 
  With the modified limiting, no oscillations are observed on any meshes we tested.
\begin{figure}[bthp]
  \centering
  \includegraphics[width=2.2in]{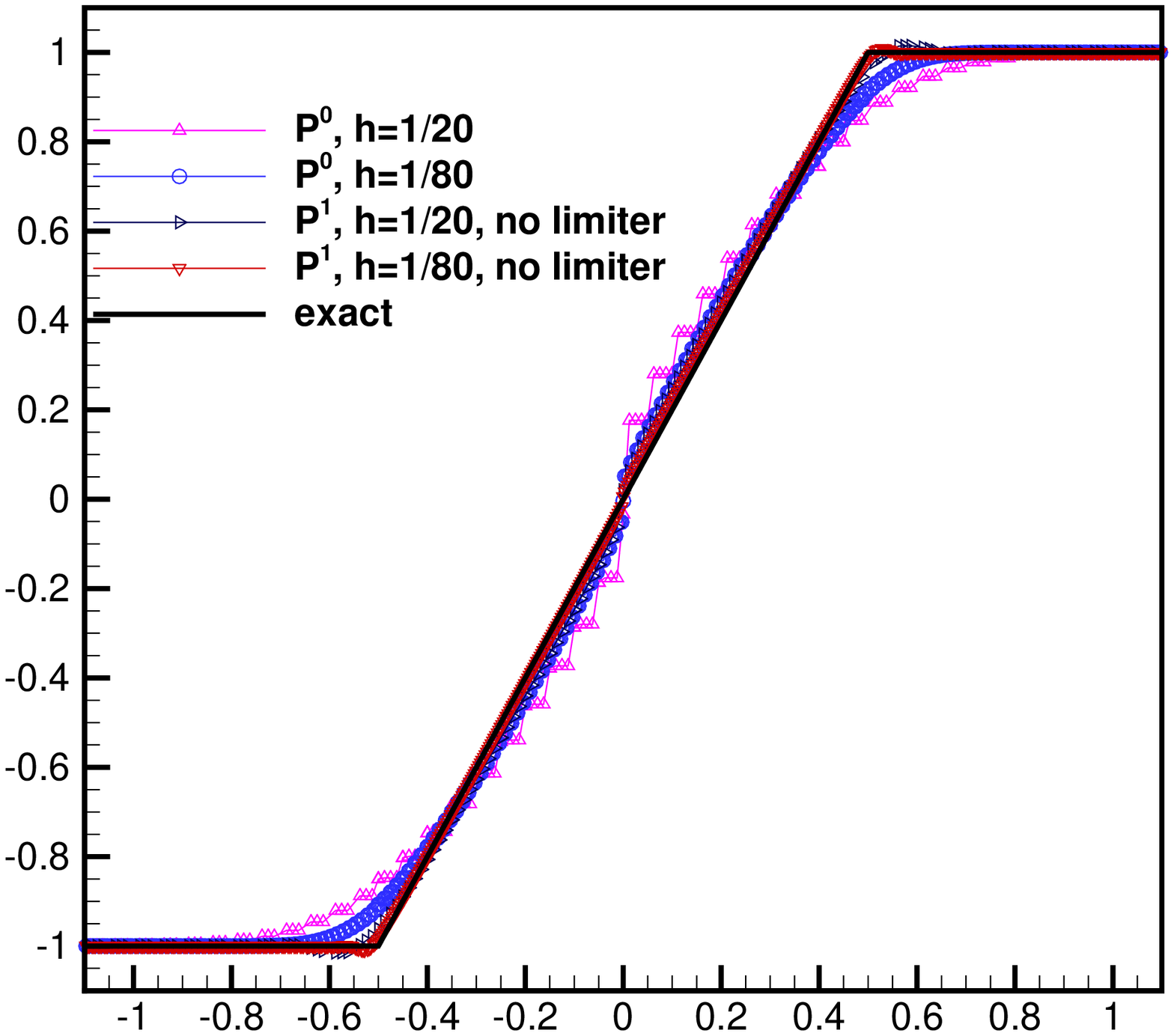}
  \includegraphics[width=2.2in]{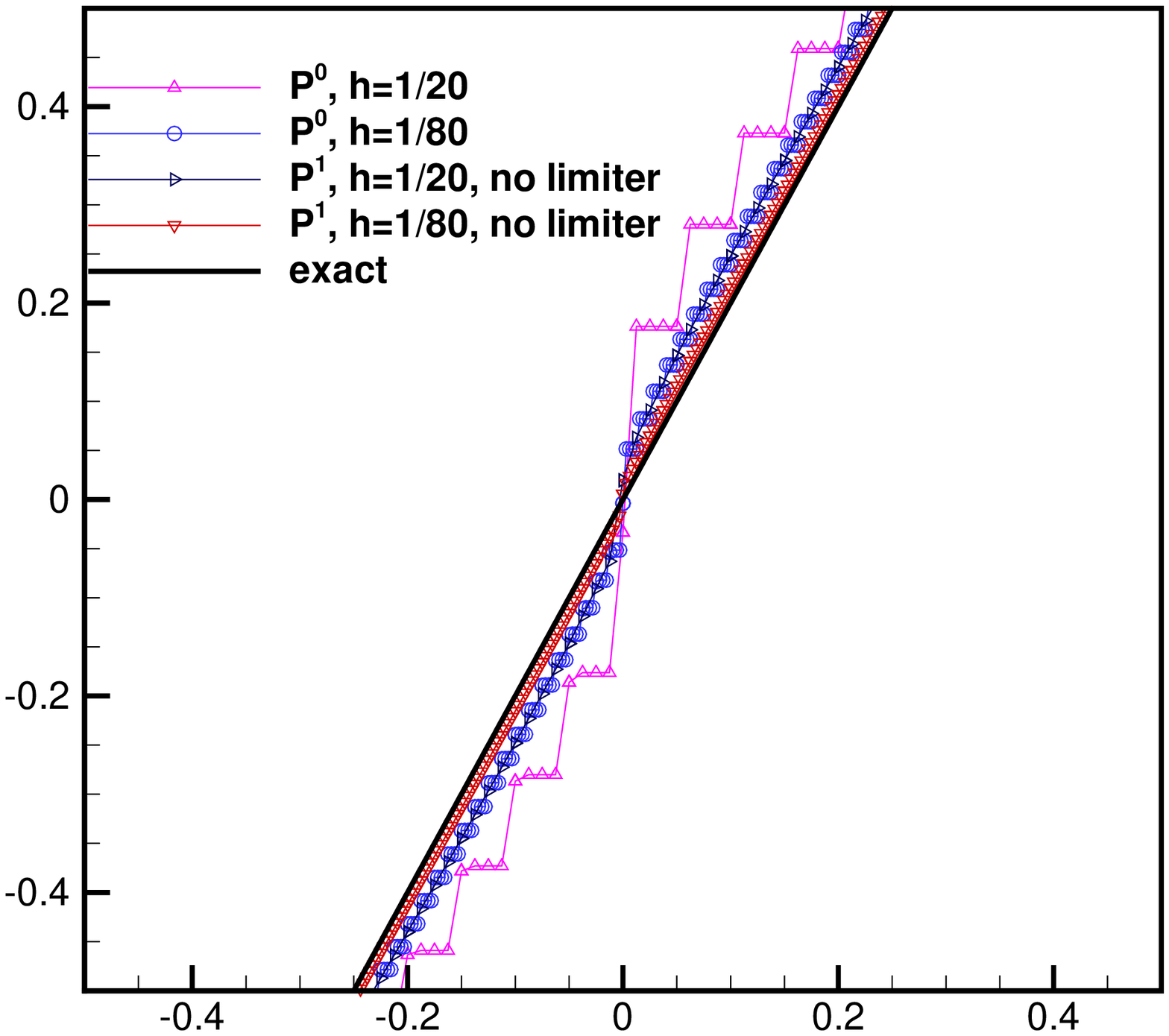}
   \includegraphics[width=2.2in]{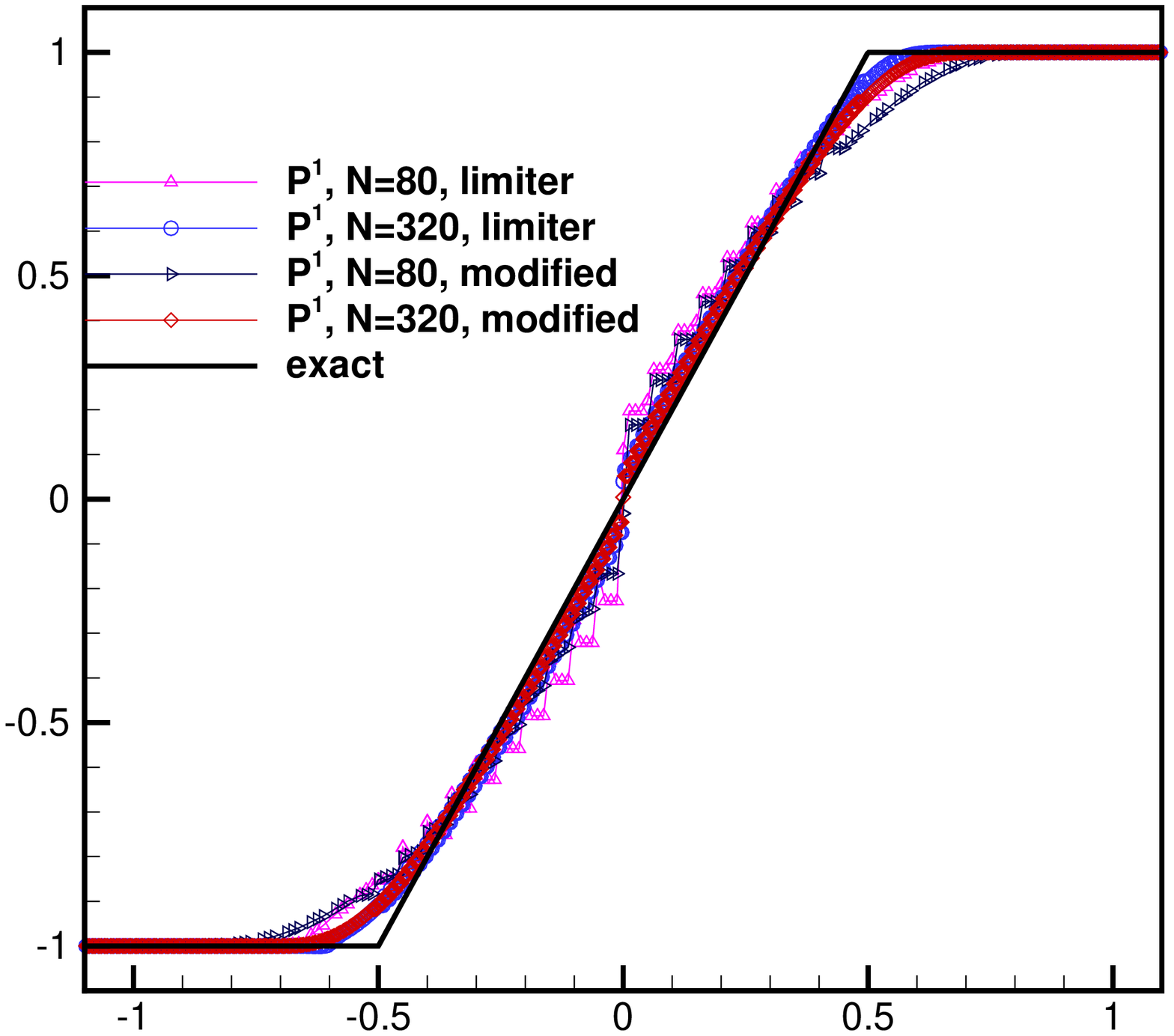}
  \includegraphics[width=2.2in]{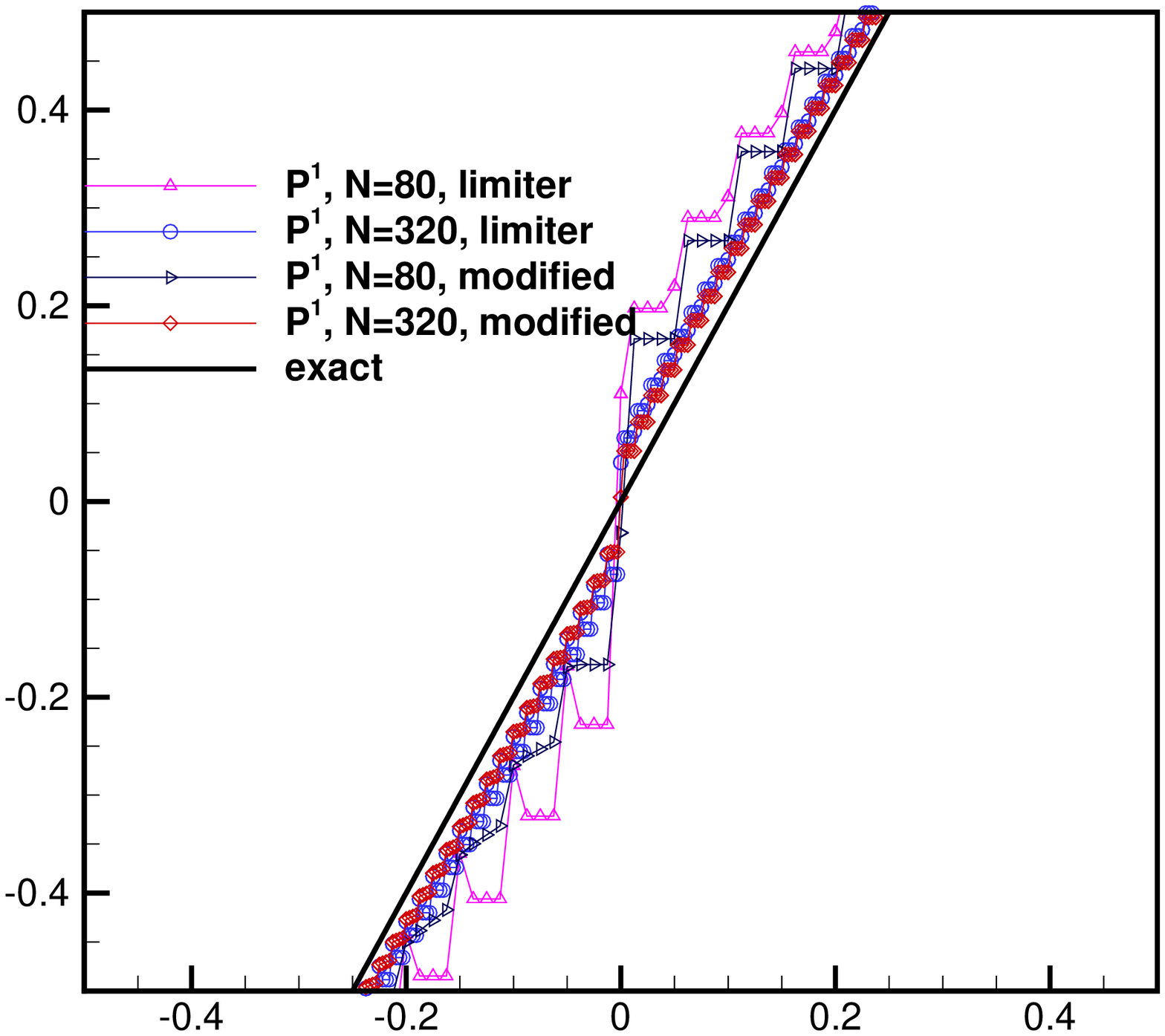} 
  \caption{Numerical solutions at $t=0.5$ for the Burgers'  equation Riemann problem with a rarefaction solution, based on $P^0$ or on $P^1$, with different limiting options. The figures to the right are magnifications of the left figures. All elements in $[-0.5,0.5]$ are cut. }\label{figure:Rarefaction}
\end{figure}

Next we consider the case with initial data $u_l=1>0>u_r=-0.5$, for which the exact solution has a shock moving with speed $v=1/4$. We solve using our cut DG scheme based on $P^0$ and $P^1$ polynomials up to time $t=0.5$ and $t=4$. The results are shown in Fig \ref{figure:RiemannShock}. We can observe that the piecewise constant approximation can capture the shock, and we did not observe any overshoots or undershoots up to $t=4$.  For the $P^1$ approximation at $t=0.5$, we observe some undershoots on the coarser meshes, which disappear when we apply the modified limiter. At time $t=4$, the shock has passed all the cut elements as well as their neighbours, and we see no undershoots or oscillations  even on the coarser meshes when standard limiting is applied.
\begin{figure}[tbhp]
  \centering
  \includegraphics[width=2.2in]{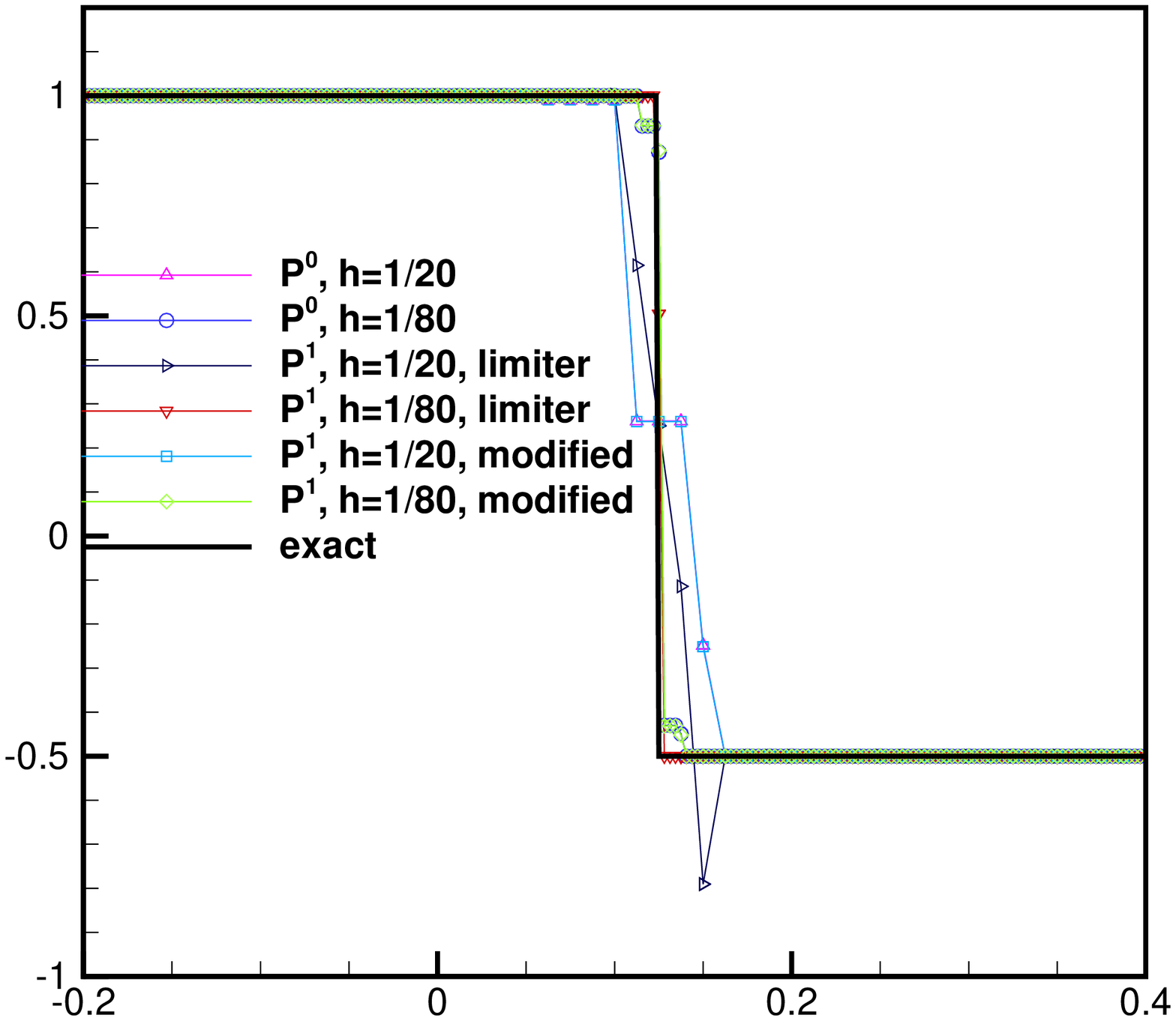}
  \includegraphics[width=2.2in]{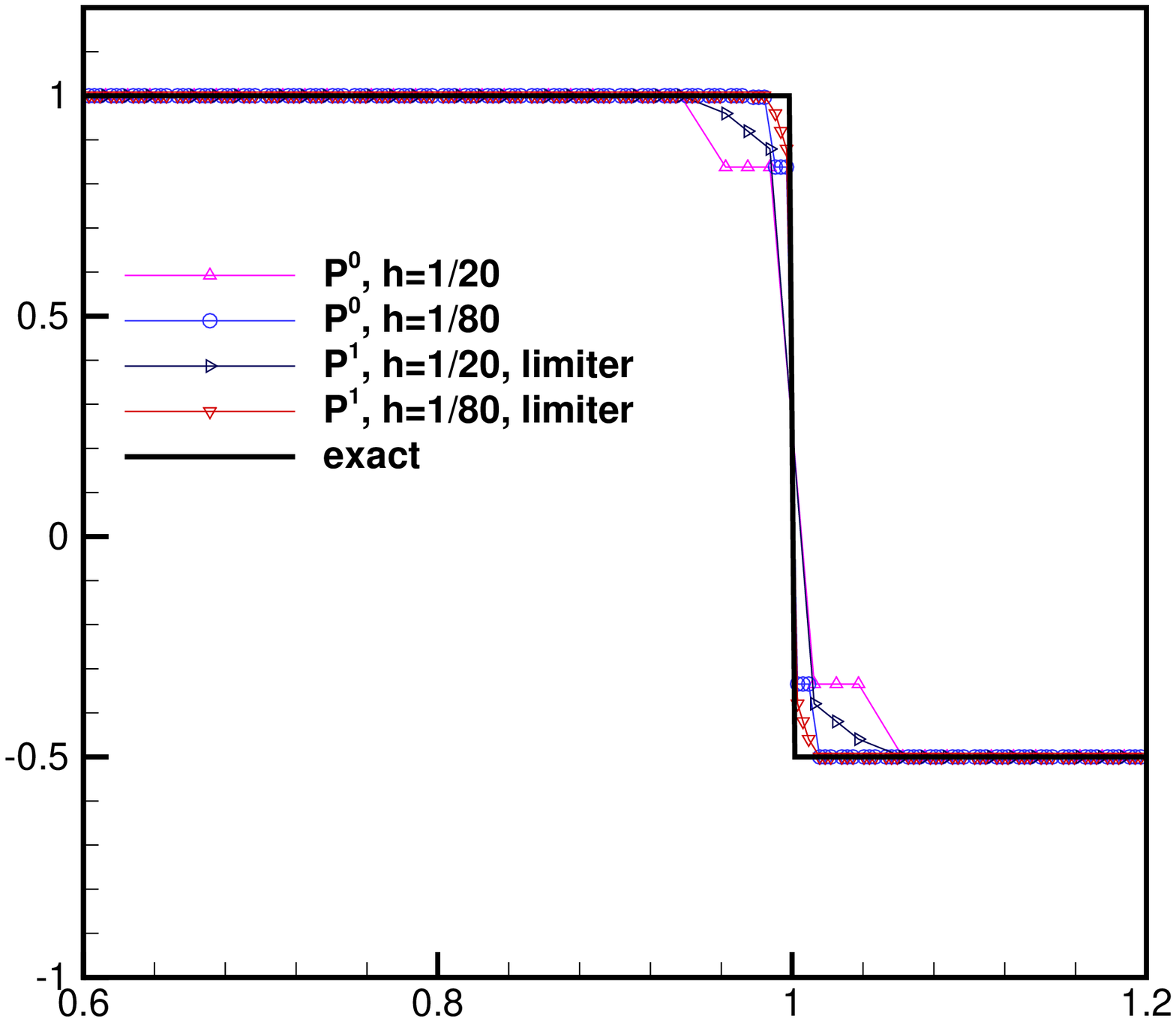}
  \caption{Numerical solutions of the Burgers' equation Riemann problem with a shock wave solution, based on $P^0$ or on $P^1$ with different limiting options. All elements in $[-0.5,0.5]$ are cut. Left: $t=0.5$, right: $t=4$.}\label{figure:RiemannShock}
\end{figure}

\section{Conclusions and future work}
\label{sec:conclusion}

We have developed a stabilized cut DG method of different orders for scalar first order hyperbolic problems in one space dimension. To avoid severe time-step restrictions or temporal instability the method includes jump stabilization at element interfaces adjacent to cut elements.  Theoretical results include $L^2$ stability for the semi-discrete  method independently of how small the cut elements are, and an accuracy result  for the linear problems, based on the unfitted $L^2$ projection. For discontinuous solutions TVD stability is essential, and we prove that the stabilized cut DG method based on piecewise constants is TVD.
From an analysis of the eigenvalues of the spatial discretization, we expect  similar Courant numbers as for the standard DG method.

Numerical experiments further investigate the properties of the methods, and demonstrate several important features.  A first result is that the CFL conditions for our cut DG methods are very similar to the CFL conditions of the corresponding standard DG methods. Secondly, a series of computations demonstrate that for smooth solutions we obtain optimal accuracy, even though the corresponding theoretical result is weaker by half an order. Thirdly,  we could observe the TVB property numerically when a TVB or a TVD limiter is applied to our scheme. 
With standard limiting  overshoots are observed as a discontinuity or shock passes a cut element or its neighbours, when higher order polynomials  are used.
Finally, we propose a modified limiting procedure involving the stabilization, which removes these artefacts.

The analysis and computations presented here can be extended to problems with interfaces where coefficients change abruptly. Of particular importance are the conservation properties at such interfaces, and we have started the investigation, and will present results elsewhere. Further analysis also includes investigating more robust limiters to control oscillations on the cut elements and if a cut element and the corresponding stabilization introduces extra dissipation, compared to the standard DG method.
 The cut DG method in this paper will also be extended to systems and to multiple dimensions in future work. We plan to consider standard cartesian elements, which are allowed to be arbitrarily cut by boundaries and/or interfaces, together with ghost stabilization for both mass and stiffness matrix. We expect, as for second order wave equations (see \cite{sticko2016higher}), that the approach will yield $L^2$ stability, high order accuracy, and reasonable CFL-numbers also for systems, and for problems in multiple space dimensions. The positive experience reported in this paper, of applying standard DG techniques, gives reason to believe that the  approach can also be applied to the extensions.
\appendix
\section{The condition numbers and  eigenvalues from standard DG method}
In Table \ref{table:uniform}, we  give the condition number of mass matrix ${\mathcal{M}}$, and maximal absolute value and maximal real part of the eigenvalue $v_i$ of the spatial operator (${\mathcal{M}}^{-1}{\mathcal{S}}$) from the standard DG scheme. It is used to compared the results from the stabilized cut DG scheme.
\begin{table}[tbhp]
\caption{\label{table:uniform} {Condition number of  mass matrix ${\mathcal{M}}$ and maximal absolute value and real part of the eigenvalue $v_i$ of the spatial operator (${\mathcal{M}}^{-1}{\mathcal{S}}$) in the standard DG scheme on the uniform mesh without cut elements (N=7 elements on the domain $[0,2]$).
} }
\begin{small}
\begin{tabular}{|c|ccc|}
\hline
degree	&	$\mathcal{K(M)}$ & $\max(|v_i|)$	&	$\max(Re(v_i))$\\
 \hline
$P^0$	&	1.00E+00	&	6.82E+00	&	1.04E-15	\\
$P^1$	&	3.00E+00	&	2.10E+01	&	2.60E-16	\\
$P^2$	&	1.13E+01	&	4.11E+01	&	2.60E-15	\\
$P^3$	&	4.38E+01	&	6.70E+01	&	9.26E-16	\\
$P^4$	&	1.72E+02	&	9.67E+01	&	-1.18E-15	\\
\hline
\end{tabular}
\end{small}
\end{table}

\section{Modified limiting for the total variation stability}\label{sec:modifiedtvd}
In this appendix, we describe a modified stabilized cut DG scheme for  problems with discontinuities. For simplification, the forward Euler method is used to present the scheme. Starting with $u^n_h$  compute $u_h^{n+1}$ by the following steps.
\begin{itemize}
\item The \textit{minmod} limiter is used as an indicator to check if the cut element or its neighbours are near a discontinuity or not.
\item If no,  we use the stabilized cut DG \eqref{scheme:cutDG2} scheme without modification to update the solution $u_h^{n+1}$. 
\item If  yes, the cut element or its neighbour needs to be limited and we define the limited solution $u_h^{(mod)}|_{I_J}=\bar{u}_h|_{I_J}$. 
Then, we modify the stabilization and polynomial space in the stabilized cut DG scheme with $P^0$ approximation on the cut element and its neighbour elements which need be stabilized, which is
\begin{align}
 \int_{I_J}u_h^{n+1}vdx+\rM h[u_h^{n+1}]_{Jl}v_{Jl}=\int_{I_J}u_h^{mod,n}vdx+\rM h[u_h^{mod,n}]_{Jl}v_{Jl}\notag\\
 \quad-\widehat{f}({u}_h^{mod,n}(x_{Jr},t))v^-_{Jr}+\widehat{f}({u}_h^{mod,n}(x_{Jl},t))v^+_{Jl}-\rA [u_h^{mod,n}]_{Jl}v_{Jl},\forall v\in V_h^0,\notag\\
 \int_{I_K}u_h^{n+1}vdx-\rM h[u_h^{n+1}]_{Jl}v_{Jl}=\int_{I_K}u_h^{mod,n}vdx-\rM h[u_h^{mod,n}]_{Jl}v_{Jl}\notag\\
 -\widehat{f}({u}_h^{mod,n}(x_{Kr},t))v^-_{Kr}+\widehat{f}({u}_h^{mod,n}(x_{Kl},t))v^+_{Kl}\notag\\
 \quad+\rA [u_h^{mod,n}]_{Jl}v_{Jl}-\int_{I_K} f(u_h^{mod,n})vdx, \forall v\in V_h^0.
 \label{scheme:tvd:limitcut}
\end{align}
Here, $I_J$ is the cut element.  We assume  $I_K$  is its neighbour element which need to stabilized and $x_{JL}$ is the interior interface between $I_J$ and $I_K$.
\item Update numerical solution $u_h^{n+1}$. Here, $u_h^{n+1}$ on elements $I_J,I_K$ are constants and $u_h^{n+1}$ is  piecewise high order polynomial on the other elements.
\end{itemize}
We note that we just modify the stabilized cut DG scheme on the cut element and its neighbours when the limiting is needed at the cut element or its neighbours. It won't change the scheme in the problems with continuous solution.

\renewcommand\refname{Reference}
\bibliographystyle{abbrv}
\bibliography{CutDG}

\end{document}